\documentclass[11pt,a4paper]{article}

\usepackage{epsfig}
\usepackage{graphicx}
\usepackage{hyperref}
\usepackage{amsmath,amsthm,amssymb,latexsym}
\usepackage{amsfonts}
\usepackage[american]{babel}
\usepackage{epsfig}

\usepackage{hyperref}
\usepackage{amsmath,amsthm,amssymb,latexsym}
\usepackage{amsfonts}
\usepackage[american]{babel}
\usepackage{mathrsfs}

\newtheorem{theorem}{Theorem}[section]
\newtheorem{definition}[theorem]{Definition}
\newtheorem{lemma}[theorem]{Lemma}
\newtheorem{cor}[theorem]{Corollary}
\newtheorem{rem}[theorem]{Remark}
\newtheorem{proposition}[theorem]{Proposition}

\newtheorem{remark}[theorem]{Remark}
\newtheorem{example}[theorem]{Example}

\newcommand{\co}{{\mathrm{co}}\,}

\newcommand{\relint}{{\mathrm{relint}}\,}

\newcommand{\R}{{\Bbb R}}

\newcommand{\Z}{{\Bbb Z}}

\def\j1n{j=1,\dots,n}
\def\j1m{j=1,\dots,m}
\def\i1np1{\in +1}

\def\R{\mathbb{R}}

\def\rn{\mathbb{R}^n}

\def\i1np1{\in +1}

\def\R{\mathbb{R}}

\def\rn{\mathbb{R}^n}

\def\u1{u^{(1)}}

\def\h1{h^{(1)}}

\newcommand{\la}{\lambda}

\newcommand{\ga} {\gamma}
\newcommand{\Ga} {\Gamma}

\newcommand{\si}{\sigma}

\newcommand{\Om} {\Omega}

\newcommand{\RR}{{\mathbb R}}

\newcommand{\pa}{\partial}

\begin{document}

\begin{title}{Bounding regions to plane   steepest descent  curves of quasi convex families
\footnote{This work has been partially supported by INdAM-GNAMPA (2014).}}
\end{title}
\author{Marco Longinetti\footnote{Marco.Longinetti@unifi.it, Dipartimento DEISTAF,
Universit\`a degli Studi di Firenze, P.le delle Cascine 15,  50144
Firenze - Italy}\\
Paolo Manselli\footnote{Paolo.Manselli@unifi.it, Dipartimento di Matematica "U. Dini, P.zza Ghiberti 27,  50122
Firenze - Italy}\\
Adriana Venturi\footnote{Adriana.Venturi@unifi.it, Dipartimento DEISTAF,
Universit\`a degli Studi di Firenze, P.le delle Cascine 15,  50144
Firenze - Italy}}
\date{}
\maketitle

\begin{abstract} Two dimensional steepest descent curves (SDC) for a quasi convex family  are considered; the problem
of their extensions (with constraints) outside of a 
convex body $K$ is studied. 
It is shown that possible extensions are constrained to lie inside of suitable bounding regions depending on $K$.
These regions are bounded by arcs of involutes of $\pa K$ and satisfy many inclusions properties.
The involutes of the boundary of an arbitrary plane convex body  are defined and written by their support function.
Extensions SDC of minimal length are constructed. Self contracting sets (with opposite orientation) are considered, 
necessary and/or sufficients conditions for them to be subsets of a SDC are proved.
\end{abstract}
\maketitle

\section{Introduction}
Let $u$ be a smooth  function defined
in a convex body $\Om\subset \RR^n$.  Let $Du(x)\neq 0$ in $\{ x\in \Om:  u(x) > \min u\}$.
A classical steepest descent curve of $u$ is  a rectifiable curve
$s\to x(s)$  solution to
$$ \dfrac{dx}{ds}= \dfrac{Du}{|Du|}(x(s)) .$$
 Classical steepest descent curves are the integral curves of a unit   field normal
 to the sublevel sets  of the given function $u$. We are interested in steepest descent curves 
 that are integral curves to a unit field normal to the family  $\{\Om_t\}:=\{x: u(x)\leq t\}$
 of the sublevel sets  for a quasi convex function $u$ (see Definition \ref{defstratifications}); $\{\Om_t\}$
 will be called a quasi convex family as in \cite{fenc}. 
 Sharp bounds about the length of the steepest descent curves for a quasi convex family, have
 been proved in   \cite{Langetepe},\cite{Mainik},\cite{Manselli-Pucci}.
 The geometry of these curves, equivalent definitions, related questions and  generalizations have
 been studied in \cite{Bolte}, \cite{Daniilidis3},  \cite{Daniilidis}, \cite{MLV}, \cite{selfdual}.

In the above works, it has beeen proved that 
{\em steepest descent curves for a quasi convex family} (SDC) can be characterized as bounded oriented
rectifiable curves $\ga \subset \R^n$,
with a locally lipschitz continuous parameterization $T\ni t \to x(t)$ satisfying 
\begin{equation}\label{defselexpwithder}
 \langle \dot{x}(t), x(\tau)-x(t) \rangle \leq 0 , \quad     \mbox{a.e.}\quad t\in T, \quad \forall \tau\leq t;
\end{equation}
$\langle \cdot, \cdot\rangle$ is the scalar product in $\RR^n$.
Let an  ordering $\preceq $ be chosen on $\ga$, according to the orientation; let us denote
 \begin{equation}\label{defgax}
\ga_x=\{ y\in \ga:  y\preceq x  \}.
\end{equation}
The curves $\ga$ satisfying \eqref{defselexpwithder} are SDC for the related quasi convex family
   $\Om_t:=co (\ga_{x(t)})$, 
 where  $co(A)$ denotes the convex hull of the set $A$.

The SDC could also be chacterized in an equivalent way as {\em self-distancing curves}, namely oriented ($\preceq$) continuous curves
with the property that  the distance of  $x$ to an arbitrarily fixed previous point 
$x_1$   is not decreasing:
\begin{equation}\label{expandingproperty}
\forall x_1,x_2,x_3\in \ga\, , x_1\preceq x_2 \preceq x_3  \Rightarrow |x_2-x_1| \leq |x_3-x_1|.
\end{equation}
In \cite{MLV} self-distancing curves are called self-expanding curves.
With the opposite orientation these curves have been also  introduced, studied 
and called self-approaching curves (see \cite{Langetepe}), or self-contracting curves (see \cite{Daniilidis}).

An important property that will be used later is the property of distancing from a set $A$:
\begin{definition}\label{defsecA} Given a set $A$, an absolutely continuous curve $\ga$, $T\ni t \to x(t)$
has the  distancing from $A$ property  if it satisfies
\begin{equation}\label{defselexpwithderfromK}
 \langle \dot{x}(t), y-x(t) \rangle \leq 0 , \quad  \quad \mbox{a.e.}\quad t \in T, \mbox{\quad and \quad}\forall y\in A.
\end{equation} 
\end{definition}

Steepest descent curves (or self-distancing curves) $\ga$ that also satisfy the above property with respect 
to a convex set $K$ will be called $SDC_K$. Of course if $\ga$ is a SDC and $x\in\ga$
 then $\ga \setminus \ga_x$ is a $SDC_{\co(\ga_x)}$.

 In the present work we are interested  on the behaviour and properties of a {\bf plane} SDC $\ga$ beyond its final point $x_0$. 
  The principal goal  of the paper is to show that conditions \eqref{defselexpwithder} 
  or \eqref{expandingproperty} imply constraints
 for possible extensions of 
 the curve $\ga$ beyond $x_0$; these constraints are written as   bounding regions for the possible extensions of $\ga_{x_0}$.
 
 Let us outline the content of our work. In \S 2 introductory definitions are given and covering maps 
 for the boundary of a plane convex
 set, needed for later use, are introduced. In \S 3 the involutes of the boundary of a plane convex body are introduced and 
 some of their properties are proved.
 
 In \S 4  plane regions depending on the convex body $co(\ga_{x_0})$ have been defined; these regions
 fence in or fence out the possible extensions of $\ga_{x_0}$.
 The boundary of these sets
  consists of arcs of involutes of convex bodies, constructed in \S 3. 
 As an application, in \S 4.1 the following problem has been studied: given a convex set $K$, $x_0\in \pa K$, $x_1\not \in K$
 is it possible to construct a $SDC_K$ joining $x_0 $ to $x_1$?  Minimal properties of this connection have been introduced
 and studied. In \S 5 sets of points more general than SDC are studied . 
 A set $\si\subset \RR^2$ (not necessarily a curve) of ordered points satisfying \eqref{expandingproperty} 
 will be   called {\em self-distancing set}, see also Definition \ref{defselfexpandinglinearly}; 
 with the opposite order, $\si$  was  called self-contracting  in  \cite{Daniilidis3} and many 
 properties of these sets, as only subsets of self contracting curves,  were  there obtained.
 A natural question arises:  does it exist a steepest descent curve $\ga \supset \si$? 
 Examples,  necessary and/or  sufficient conditions are given when 
 $\si$ consists of a finite or countable number of points $x_i \in \RR^2$  and/or 
  steepest descent curves $\ga^i \subset \RR^2$.  
 
 In the present work  the two dimensional case is studied. Similar results for the  $n$ dimensional case are
 an open problem stated at the end of the work.

\section{Preliminaries and definitions}
Let
$$
B(z,\rho)=\{x\in\rn\,:\,|x-z|<\rho\}\,,\quad\,
S^{n-1}=\partial B(0,1) \, \quad n\geq 2.
$$
A not empty, compact convex set $K$ of $\rn$ will be called  a {\em
convex body}. From now on, $K$ will {\bf always be a convex body not reduced to a point}. 
$Int(K)$ and $\pa K$ denote the interior of $K$ and
the boundary of $K$, $|\pa K|$  denotes its length, $cl(K)$  is    the closure of $K$,
$\mathit{Aff}(K)$ will be the smallest affine space containing $K$;
 $\relint K$ and  $\pa_{rel} K$ are  the corresponding subsets in
the topology of $\mathit{Aff}(K)$.
 For every set $S \subset\mathbb{R}^n$, $co(S)$ is the convex hull of $S$. 

Let $q\in  K$; the {\em normal cone } at
$q$ to $K$ is the closed convex cone

\begin{equation}\label{normalcone} N_K(q)=\{x\in\rn: \langle x,y-q\rangle \le 0 \quad
\forall y \in K\}.
\end{equation}
 When $q \in Int(K)$, then  $N_K(q)$ reduces to zero.

The  {\em tangent cone}, or  support cone, of  K at a
point $q \in \pa K$  is given by
$$ T_K(q)=cl \left(\bigcup_{y\in K} \{s( y-q): s \geq 0  \}\right).$$

In two dimensions cones will be called sectors.

Let $K$ be a convex body  and $p$ be a point. A  simple
cap body $K^p$ is:
\begin{equation}\label{capbody}
K^p=\bigcup_{0\le \lambda \le 1}\{ \la K+(1-\la)p\}=co(K\cup\{p\}).
\end{equation}
Cap bodies properties can be found in \cite{Bonnfen},\cite{Schn}.

\subsection{Self-distancing sets and steepest descent curves}
Let us recall the following definitions
\begin{definition}\label{defselfexpandinglinearly}
Let us call  {\bf self-distancing  set } a bounded subset $\si$ of $\RR^n$,    linearly  ordered 
 (by $\preceq$),
 with the property:
\begin{equation}\label{IBDCmonlinearly}
x_1, x_2,
x_3 \in \si, \; \mbox{ and } \; x_1 \preceq x_2 \preceq x_3 \quad
\Longrightarrow \quad  |x_2-x_1| \leq |x_3-x_1|.
\end{equation}
\end{definition}
The self-distancing sets has been introduced in \cite{Daniilidis3} with the opposite order.
If a self-distancing  set $\si$ is a closed connected set, not reduced to a point,  then it  can be proved that 
$\si$ is a steepest descent  curve (see 
\cite[Theorem 3.3]{Daniilidis3}, \cite[Theorem 4.10]{MLV}) and it 
will also be called  a {\bf self-distancing curve}  $\ga$. 

The short name SDC will be used for self-distancing curves (steepest descent curves) in all 
the paper.
 \begin{definition}\label{defSEC+} Let $K$ be a  convex body,
$\ga\subset \R^2\setminus \relint K$
 will be  called  a {\bf self-distancing    curve from $K$} (denoted $SDC_K$) if:
\begin{enumerate}
\item[(i)] $\ga$ is  a self-distancing  curve,
\item[(ii)] $\ga \cap \pa_{rel} K \neq \emptyset$,
\item[iii)] $\ga$ has the {\bf distancing  from $K$} property:
\begin{equation}\label{ineqdefEC+}
\forall y \in K, \forall  x, x_1 \in \ga: x \preceq x_1 \Rightarrow |x-y| \leq |x_1-y|.
\end{equation}
\end{enumerate}
When (ii) does not hold, that is $\ga \cap \pa_{rel} K =\emptyset$, $\ga$ will be called a {\bf deleted 
self-distancing    curve from $K$}.
\end{definition}
\begin{rem}\label{remga^K} Let $\ga$ be a $SDC_K$, then  $\ga$   has an absolutely continuous parameterization, thus
property \eqref{ineqdefEC+} for $\ga$ 
is equivalent to \eqref{defselexpwithderfromK}.
\end{rem}
Nested families of convex sets have been introduced and studied by De Finetti \cite{Defi} and Fenchel
\cite{fenc}. Let us recall some definitions.
\begin{definition}\label{defstratifications} Let $T$ be a real interval.
A {\bf convex stratification} (see \cite{Defi}) is  a not empty family $\mathfrak{K}$ of  
convex bodies $\Om_t\subset \rn$, $t\in T \subset \RR$, linearly
strictly  ordered by inclusion ($\Om_{1} \subset \Om_{2} $,
$\Om_{1} \neq \Om_{2} $), with a maximum set ($\max\mathfrak{K}$) and a minimum set
($\min\mathfrak{K}$).

Let $\mathfrak{K}=\{\Om_t\}_{t\in T}$ be a convex stratification. 
If for every
 $s \in T\setminus \{\max T \}$ the property:
$$\bigcap_{t>s}\Om_t=\Om_{s}$$
holds, then
as in \cite{fenc}, $\mathfrak{K}=\{\Om_t\}_{t\in T} $ will be called a {\bf quasi convex family}.
\end{definition}

An important  quasi convex family  associated to a continuous
self-distancing  curve from $K$, $\ga$: $t\to x(t)$ is $\mathfrak{K}=\{\Om_t\}_{t\in T},$ where
$$\Om_t=co(\ga_{x(t)}\cup K).$$
The couple $(\ga,\mathfrak{K})$ is  special case of {\em Expanding Couple}, 
a class introduced in \cite{MLV}.

\begin{remark}\label{sdcofprevious}
 If $\ga\in SDC_K$, then for  all $x \in \ga$ the curve $(\ga\setminus \ga_x)\cup \{x\}$ is a self-distancing curve from the 
 convex hull of the set $\ga_x \cup K$.
\end{remark}

This fact is a direct consequence of the following 
\begin{proposition}[\cite{MLV}, Lemma 4.9 ]\label{lemmaecissec}  Let  $p,q,y_i\in \rn,  \,  i=1, \ldots,s  $. If
\begin{equation}\label{axesineq}
|p-y| \leq |q-y|, \quad \mbox{for \quad } y=y_i, \,  i=1, \ldots,s
\end{equation}
then the same holds for every  $y\in co(\{y_i, i=1,\ldots, s\})$. In \eqref{axesineq}  inequality $\leq$
can be changed with strict inequality.
\end{proposition}

\subsection{The support function of a plane convex body}\label{2.2}
Let $K\subset \RR^n$ be a convex body not reduced to a point.

For  a convex body $K$, the {\em support function} is
defined as

\begin{equation*}
\label{equation1} H_K(x)=\sup_{y\in K}\langle x,y\rangle\,,\quad
x\in\mathbb{R}^n,
\end{equation*}
where $\langle \cdot, \cdot \rangle$ denotes the  scalar
product in $\mathbb{R}^n$. For $n=2$, $\vartheta \in \RR$, let $\theta=(\cos \vartheta, \sin \vartheta)\in S^1$
and $h_K(\vartheta):=H_K(\theta)$, it will be denoted 
 $h(\vartheta)$ if no ambiguity arises. 

For every $\theta\in S^1$ there exists at least  one point $x\in \pa K$ such that:
\begin{equation}\label{defFtheta}
 \langle \theta, y-x\rangle \leq 0 \quad \forall y \in K;
 \end{equation}
this means that the line through $x$ orthogonal to  $\theta$ supports $K$.
For every $x\in \pa K$ let $\widehat{N_x}$ the set of $\theta \in S^1$ such that \eqref{defFtheta}
holds. Let $F(\theta)$ be the set of all $x\in \pa K$ satisfying \eqref{defFtheta}. If
$\pa K$ is  strictly convex at the direction $\theta$ then 
$F(\theta) $ reduces to one point and  it will be denoted by $x(\theta)$.
 \begin{definition}\label{defGF} The set valued map: $G: \pa K \to S^1, \pa K \ni x \to \widehat{N_x}\subset S^1$, 
 is the generalized Gauss map;
$x\in \pa K$ is a vertex on $\pa K$ iff  $\widehat{N_x}$ is a sector with interior points. 
The set valued map $F: S^1 \to \pa K, S^1\ni \theta \to F(\theta) \subset \pa K$ is the reverse generalized Gauss map;
$F(\theta)$  is a closed segment, possibly reduced to a single point, 
and it will be called 1-face when it has interior points.
  \end{definition}

 Let $P$ be the covering map 
 $$P:\RR \to S^1, \RR \ni \vartheta \to \theta=(\cos \vartheta, \sin \vartheta) \in S^1.$$

Let $L=|\pa K|$,
let $s \to x_l(s), 0 \leq s < L$ ($s \to x_r(s), 0 \leq s < L$) 
be the parametric representations of $\pa K$  depending on the arc length counterclockwise (clockwise) with
an initial point (not necessarily the same). Let us extend $x_l(\cdot)$ and $x_r(\cdot)$ by defining
$$x_l(s):=x_l(s-kL) \quad \mbox{if} \quad kL\leq s < (k+1)L, \quad (k\in \Z).$$
Similarly for $x_r$.

Let us fix $x_0\in \pa K$,  $\theta_0 \in G(x_0)$, $\theta_0=(\cos \vartheta_0, \sin \vartheta_0)$, $\vartheta_0\in \RR$.

For later use, we need to have $x_0=x_l(s_0)=x_r(s_0)$; this can be realized by choosing suitable initial points for 
the parameterizations $x_l$ and $ x_r$.

Then
$$x_l(s_0+s)=x_r(s_0+L-s), \, \forall s\in \RR.$$ 

The maps 
$$x_l: \RR \to \pa K, \quad x_r:\RR \to \pa K, $$
are covering maps. 

  The initial parameters will be
  $$x_0=x_l(s_0)=x_r(s_0) \in \pa K, S^1\ni \theta_0 \in \stackrel{-1}{F}(x_0), \RR \ni \vartheta_0\in \stackrel{-1}{P}(\theta_0),$$
  ( $\stackrel{-1}{F}(x_0),\stackrel{-1}{P}(\theta_0)$ are the back images   of $F,P$ respectively).
Let  $k\in \Z$.
Let us define, for $\vartheta_0+2k\pi <\vartheta < \vartheta_0+2(k+1)\pi$:
\begin{equation}\label{defsl+}
 s_{l+}(\vartheta):=\sup \{s\in \RR: kL< s \leq (k+1) L,\, x_l(s)\in F(P(\vartheta))\};
 \end{equation}
 if $\vartheta=\vartheta_0+2k\pi $
 \begin{equation}
 s_{l+}(\vartheta):=\sup \{s\in \RR: kL\leq  s < (k+1) L,\, x_l(s)\in F(P(\vartheta))\}.
\end{equation}
Similarly, let us define for $\vartheta_0+2(k-1)\pi <\vartheta < \vartheta_0+2k\pi$:
\begin{equation}
\label{defsr-}
 \quad s_{r-}(\vartheta):=\inf \{s\in \RR: kL\leq s < (k+1) L,\, x_r(s)\in F(P(\vartheta))\};
 \end{equation}
 if $\vartheta=\vartheta_0+2k\pi $
 \begin{equation}
  \quad s_{r-}(\vartheta):=\inf \{s\in \RR: (k-1)L< s \leq k L, \, x_r(s)\in F(P(\vartheta))\}.
\end{equation}
The function $s_{l+}$ is increasing in $\RR$, right continuous and with left limits (so called {\em cadlag} function).
Similar properties hold for $-s_{r-}$.
Let us recall that a cadlag increasing function $s(\vartheta)$, $\vartheta\in \RR$ has a right continuous inverse defined as
$$\vartheta(s)=\inf\{\vartheta: s(\vartheta)> s\}.$$
Let $ \vartheta_{l+}(\cdot)$ the right continuous inverse of $s_{l+}(\cdot)$. 
Let $s \to\vartheta_{r-}(s)$ the opposite of the right continuous inverse of $-s_{r-}(\cdot)$.

Let us introduce for simplicity
$${\bf n}_{\vartheta}:=(\cos\vartheta, \sin\vartheta),\, {\bf t}_{\vartheta}:=(-\sin\vartheta, \cos\vartheta).$$

Let $\vartheta \to h(\vartheta)$ be the support function of $K$.

It is well known (\cite{gugg}) that, if
 $\pa K$ is $C^2_+$ (that is $\pa K \in C^2$, with positive curvature), 
 then $h$ is $C^2$ and the counterclockwise element arc $ds$ of $\pa K$ 
is given by
\begin{equation}\label{ds=dtheta}
 ds=(h+\ddot{h})d\vartheta.
\end{equation}
  $h(\vartheta)+\ddot{h}(\vartheta)$ is the positive  radius of curvature;
 moreover the reverse Gauss map $F:\theta \to x \in \pa K$ 
 is a 1-1 map given by 
\begin{equation}\label{supfunctionplane}
 x(\theta):=h(\vartheta){\bf n}_{\vartheta}+\dot{h}(\vartheta){\bf t}_{\vartheta}, \quad \vartheta \in \stackrel{-1}{P}(\theta). 
\end{equation}

The previous formula also holds for an arbitrary convex body, for every $\vartheta$ such that $F(\theta)$ is reduced 
to a point, see \cite{Bonnfen}.
Let us recall that a real valued function $x\to f(x)$ is called semi convex on $\RR$ 
when there exists a positive constant $C$ such that $f(x)+Cx^2$ is convex on 
$\RR$.
From \eqref{ds=dtheta} the function $\vartheta \to h(\vartheta)+\frac12\vartheta^2\max h$ is convex on $\RR$, 
thus $h$ is semi convex. In the case that  $K$ is an arbitrary convex body, by approximation arguments with $C^2_+$ convex bodies,
see \cite{Schn},
 it follows that the support function of  $K$ is also semi convex.
 As consequence $h$ is Lipschitz continuous, it  has left (right) derivative $\dot{h}_{-}$ (respectively $\dot{h}_{+}$)
at each point, which is left (right) continuous. 
Moreover at each point the right limit of $\dot{h}_{-}$ is $\dot{h}_{+}$ and 
the left limit of $\dot{h}_{+}$ is $\dot{h}_{-}$, see \cite[pp. 228]{Rock}.

It is not difficult to show (from \eqref{supfunctionplane}, with a right limit argument) 
that for an arbitrary convex body,  for 
$ \vartheta \in \RR$, the formula
\begin{equation}\label{supfunctionplane+}
 x_l(s_{l+}(\vartheta))=h(\vartheta){\bf n}_{\vartheta}+\dot{h}_+(\vartheta){\bf t}_{\vartheta}
 \end{equation}
 holds. Similarly the formula
 \begin{equation}
 \label{supfunctionplane-}
  x_r(s_{r-}(\vartheta))=h(\vartheta){\bf n}_{\vartheta}+\dot{h}_-(\vartheta){\bf t}_{\vartheta}
 \end{equation}
holds.

If $\pa K$ is not strictly convex at the direction $\theta=(\cos \vartheta, \sin \vartheta)$
then $h$ is not differentiable at $\vartheta$ and 
\begin{equation}\label{Ftheta}
 \dot{h}_+(\vartheta)-\dot{h}_-(\vartheta)=|x_l(s_{l+}(\vartheta))-x_r(s_{r-}(\vartheta))|=|F(\theta) |.
\end{equation}
If $x_1,x_2\in \pa  K$ let us define  $arc^+(x_1,x_2)$ the set of points of $\pa K$ between 
  $x_1$ and $x_2$ according to the counterclockwise orientation of $\pa K$, and  $arc^-(x_1,x_2)$ the set 
  of points between $x_1$ and $x_2$, according to the 
  clockwise orientation; 
  $|arc^\pm(x_1,x_2)|$ denote their  length.
  \begin{remark}\label{snconverge} It is well known that a sequence of convex body $K^{(n)} $ converges to $K$ if and only if the corresponding 
 sequence of support functions converges in the uniform norm, see \cite[pp. 66]{Schn}. 
 Moreover as the two sequences of the end points of a closed connected arc of $\pa K^{(n)}$ converge,
 then the sequence of the corresponding arcs converges to a connected arc of $\pa K$ and the sequence of
 the corresponding lengths converges too.   
  \end{remark}

\begin{proposition} Let $K$ be a convex body and $h$ its support function, then
\begin{eqnarray}\label{vaules+conh}
 s_{l+}(\vartheta)-s_{l+}(\vartheta_0)=\int_{\vartheta_0}^{\vartheta}h(\tau)\, d\tau +\left(\dot{h}_
 +(\vartheta)-\dot{h}_+(\vartheta_0)
 \right), \quad \forall \vartheta \geq  \vartheta_0;\\
\label{vaules-conh}
 s_{r-}(\vartheta_0)-s_{r-}(\vartheta)=\int_{\vartheta_0}^{\vartheta}h(\tau)\, d\tau 
 +\left(\dot{h}_-(\vartheta)-\dot{h}_-(\vartheta_0)
 \right), \quad \forall \vartheta \leq  \vartheta_0.
\end{eqnarray} 
\end{proposition}
\begin{proof}
 For every convex body $K$ not reduced to a point the function $\vartheta \to s_{l+}(\vartheta)$ 
is  defined everywhere and satisfies the weak form of \eqref{ds=dtheta}, namely:
\begin{equation}\label{weakds=dthaeta}
 -\int_{\RR} s_{l+}(\eta)\dot{\phi}(\eta)d\eta=
 \int_{\RR}(\phi +\ddot{\phi})(\eta) h(\eta) d\eta, \quad \forall \phi \in {C}^\infty_0(\RR).
 \end{equation}
Using the fact that $\vartheta \to h(\vartheta)$ is Lipschitz continuous, integrating by parts 
\eqref{weakds=dthaeta}, the formula 
\begin{equation}\label{weakds=dthaeta+}
  -\int_{\RR} s_{l+}(\eta)\dot{\phi}(\eta)d\eta=
  -\int_{\RR}\dot{\phi}(\eta)\left(\int_0^\eta h(\tau)d \tau +\dot{h}(\eta)\right) d \eta ,\quad \forall \phi \in {C}^\infty_0(\RR)
\end{equation}
holds. Thus
$$s_{l+}(\eta)=c +\int_0^\eta h(\tau)d \tau +\dot{h}(\eta) , \quad  \mbox{a.e.} $$
with $c$ constant.
Passing to the right limit, the equality
$$s_{l+}(\eta)=c+\int_0^{\eta} h(\tau)d \tau +\dot{h}_+(\eta), \quad \forall \eta \in \RR$$
holds.
The formula \eqref{vaules+conh} follows, by computing $s_{l+}(\vartheta)-s_{l+}(\vartheta_0)$, using the previous equality.
Similarly  \eqref{vaules-conh} is proved.
\end{proof}

\section{Involutes of a closed convex curve}
\begin{figure}[htb]
\epsfig{file=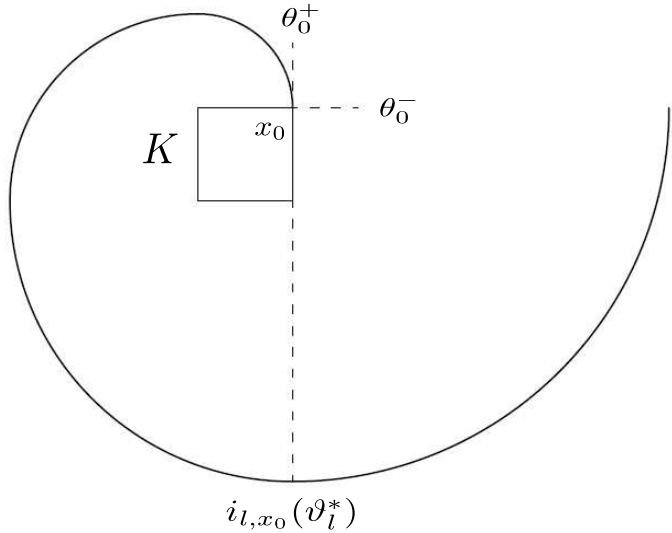, width=10cm} \caption{Left involute of a square } \label{square}

\end{figure}

\begin{definition}Let $I$ be an interval.
 A plane curve $ I\ni t \to  x(t) $ is convex if at every point $ x $
 it has right tangent vector $T^+(x)$ and
 $\arg(T^+(x(t))$ is a not decreasing function.
\end{definition}
Let $s \to x(s)$ be the arc length parameterization of a smooth curve;
the classical definition of involute  
starting at a point $x_0=x(s_0)$ of the curve $x(\cdot)$  is
\begin{equation}\label{clasdefinv}
 i(s)=x(s)-(s-s_0)x'(s) \quad s \geq s_0. 
\end{equation}

Let us notice that $s$ is the arc length of the curve, not of the involute; if $s_0=0$, then the starting point of 
the involute coincides with the starting point of the curve.
It is easy to construct  an involute of a convex polygonal line (even if the classical  definition \eqref{clasdefinv}
does  not work) by using arcs of circle centered at its corner points; moreover
the involute depends on the orientation of the curve.

In this section,  involutes for the boundary of an arbitrary  plane  convex body $K$, not reduced to a point,
will be defined. The assumption that $K$ is an arbitrary convex body is needed to work with the involutes of the 
convex sets, not smooth, studied in \S \ref{fences}.

Let $K\in C^2_+$; let $x_0$ be a fixed point of $\pa K$, $s \to  x(s) $ can be   the clockwise  parameterization of $\pa K$ or the
counterclockwise parameterization.
Since there exist
two orientations, then two different involutes have to be considered. 
As noted previously one can assume that the parameterizations of $\pa K$ have been chosen
so that  $x_0=x_l(s_0)=x_r(s_0)$.

\begin{definition} Let us denote by $i_{l,x_0}$ the left involute of $\pa K$ starting at $x_0$ corresponding  to 
the counterclockwise parameterization  of $\pa K$, by $i_{r,x_0}$ 
the right involute
 corresponding to the clockwise parameterization.
  When one  needs to emphasize   the  dependence on $K$ of  involutes, they  will be written as $i^K_{l,x_0},i^K_{r,x_0}$.
  \end{definition}
  \begin{rem}
 Let us notice that if $\rho$ is a plane reflection with respect to a fixed axis then 
 $$i^K_{r,x_0}=\rho(i^{\rho(K)}_{l,\rho(x_0)}).$$
 This relation allows us to prove our results for the 
 left involutes only and to state without proof the analogous results for the right involutes.
 \end{rem}

\begin{theorem}\label{theoreminvtheta}Let us fix the initial parameters $x_0,s_0,\theta_0,\vartheta_0$.
The left and the right involutes of a plane convex curve starting at   $x_0 \in \pa K$, 
boundary of a $C^2_+$ plane convex body $K$ with support function $h$, 
are parameterized by the value $\vartheta$ 
related to the outer
normal ${\bf n}_{\vartheta}$ to $K$, as follows
\begin{eqnarray}\label{parinvtheta}
 i_{l,x_0}(\vartheta)=h(\vartheta){\bf n}_{\vartheta}-
 \left( \int_{\vartheta_0}^{\vartheta}h(\tau)d\tau- \dot{h}(\vartheta_0)\right){\bf t}_{\vartheta}, 
 \mbox{\quad for \quad} \vartheta\geq \vartheta_0, \\
  i_{r,x_0}(\vartheta)=h(\vartheta){\bf n}_{\vartheta}-
 \left( \int_{\vartheta_0}^{\vartheta}h(\tau)d\tau- \dot{h}(\vartheta_0)\right){\bf t}_{\vartheta}, 
 \mbox{\quad for \quad} \vartheta\leq \vartheta_0. \label{parinvtheta_r}
\end{eqnarray} 
\end{theorem}
\begin{proof} In the present case there is a 1-1 mapping between $\vartheta$ and $s$; from \eqref{ds=dtheta},  it follows
$$s-s_0=\int_{\vartheta_0}^{\vartheta}h(\tau)d\tau +\dot{h}(\vartheta)-\dot{h}(\vartheta_0);$$
then, changing the variable $s$ with $\vartheta$ in \eqref{clasdefinv}, with elementary computation, 
\eqref{parinvtheta} is obtained (since
$x'(s)={\bf t}_{\vartheta}$ and \eqref{supfunctionplane} holds). Formula \eqref{parinvtheta_r} follows from
\eqref{supfunctionplane-} and \eqref{vaules-conh}. \end{proof}

For an arbitrary convex body $K$ in place of \eqref{supfunctionplane}, formulas 
\eqref{supfunctionplane+},\eqref{supfunctionplane-} have to be used.
\begin{definition}\label{theoreminvthetaest} Let $K$ be a plane convex body, 
let 
$$x_0=x(s_0) \in \pa K,  \vartheta_0^+:=\vartheta_{l+}(s_0), s_0^+:=s_{l+}(\vartheta_0^+).$$
The left involute of  $\pa K$  starting at $x_0$ will be defined as
\begin{equation}\label{defleftinvolutegeneral}
        i_{l,x_0}(\vartheta)=x_l(s_{l+}(\vartheta))-
     \left(s_{l+}(\vartheta)- s_0\right){\bf t}_\vartheta \quad \mbox{for} \quad \vartheta \geq \vartheta_0^+;
           \end{equation} 
Similarly if $\vartheta_0^-:=\vartheta_{r-}(s_0)$, $s_0^-:=s_{r-}(\vartheta_0^-)$,
the right involute starting at $x_0$ will be defined as
\begin{equation}\label{defrightinvolutegeneral}
        i_{r,x_0}(\vartheta)=x_r(s_{r-}(\vartheta))+
     \left(s_{r-}(\vartheta)- s_0\right){\bf t}_\vartheta \quad \mbox{for} \quad \vartheta \leq \vartheta_0^-.
           \end{equation} 
\end{definition}
 From \eqref{defleftinvolutegeneral}, \eqref{vaules+conh} it follows that
\begin{equation}\label{parinvtheta+}
 i_{l,x_0}(\vartheta)=h(\vartheta){\bf n}_{\vartheta}-
 \left( \int_{\vartheta_0^+}^{\vartheta}h(\tau)d\tau- \dot{h_+}(\vartheta_0^+)\right){\bf t}_{\vartheta}-
 |x_0-x_l(s_0^+)|{\bf t}_{\vartheta},
 \quad \vartheta \geq \vartheta_0^+;
\end{equation} 
Similarly from \eqref{defrightinvolutegeneral}, \eqref{vaules-conh} it follows that
\begin{equation}\label{parinvtheta-}
 i_{r,x_0}(\vartheta)=h(\vartheta){\bf n}_{\vartheta}-
 \left( \int^{\vartheta}_{\vartheta_0^-}h(\tau)d\tau -\dot{h_-}(\vartheta_0^-)\right){\bf t}_{\vartheta}
 + |x_0-x_r(s_0^-)|{\bf t}_{\vartheta},
 \quad \vartheta \leq \vartheta_0^-.
\end{equation}

Let us notice that in \eqref{parinvtheta+}, \eqref{parinvtheta-} 
the same parameter $ \vartheta$  is used, but with different range; it turns out that $i_l$ is 
counterclockwise oriented; instead $i_r$ is clockwise oriented;
$x_0=i_{l,x_0}(\vartheta_0^+)=i_{r,x_0}(\vartheta_0^-)$ .

\begin{rem}\label{i.ii,iii} The following facts can be derived from the above equations:
\begin{enumerate}
 \item[i)] since $h$ is Lipschitz continuous for every convex body $K$, then
 the involute $i_{l,x_0}$ is a rectifiable curve;
    \item [ii)] $i_{l,x_0}(\vartheta_0^+)=x_0$  and       
      \begin{equation}\label{remiv}
       |i_{l,x_0}(\vartheta)-x_l(s_{l+}(\vartheta))|=s_{l+}(\vartheta)-s_0;
      \end{equation}      
     \item[iii)] if $x$ is a vertex of $\pa K$ then 
 $i_{l,x_0}(\vartheta)$, for $(\cos\vartheta,\sin\vartheta) \in N_K(x)$, lies on an arc of circle centered at $x$ with 
 radius $s_{l+}(\vartheta)-s_0$ ; 
      \item[iv)]  the involute \eqref{parinvtheta} satisfies
\begin{equation}\label{parallelself}
 i_{l,x_0}(\vartheta+2\pi)=i_{l,x_0}(\vartheta)-L{\bf t}_{\vartheta}, \quad \forall \vartheta \geq\vartheta_0^+.
\end{equation}
\end{enumerate}
\end{rem}

\begin{lemma}\label{regi(theta)} The parameterization \eqref{parinvtheta+} of the involute  
$i_{l,x_0}$ is 1-1 in the interval 
$[\vartheta_0^+,\vartheta_0^++2\pi)$; 
moreover, except for at most a finite or countable set $\mathfrak{F}$ of  values $\vartheta_i$, 
$i=1,2,\ldots...$ (corresponding to the 1-faces $F_{\theta_i}$ of $\pa K$), $i_{l,x_0}$ is  differentiable and:
 \begin{equation}\label{di(theta)} 
  \frac{d}{d\vartheta}i_{l,x_0}(\vartheta) =\left(s_{l+}(\vartheta)- s_0\right){\bf n}_\vartheta\quad \mbox{for}\quad \vartheta 
  >\vartheta_0^+, \vartheta \not\in \mathfrak{F};
 \end{equation}
furthermore $i_{l,x_0}$   has left and right derivative with
   common direction ${\bf n}_{\vartheta}$ at $\vartheta= \vartheta_i\in \mathfrak{F}$.
          \end{lemma}
 \begin{proof}
  By differentiating \eqref{parinvtheta+} and using \eqref{vaules+conh}, the equality  \eqref{di(theta)} is proved.
  Similar argument, at  $\vartheta=\vartheta_i \in \mathfrak{F}$, proves that ${\bf n}_{\vartheta}$ 
  is the common direction of the left and right derivatives.
 \end{proof}
 \begin{rem}\label{parallel}
 Let $\vartheta \to i_{l,x_1}(\vartheta),\vartheta \to i_{l,x_2}(\vartheta),$ $x_i=x(s_i)$, $i=1,2$
 be left involutes of $K$. Since 
     $$i_{l,x_2}(\vartheta)-i_{l,x_1}(\vartheta)=(s_2-s_1){\bf t}_{\vartheta}, 
     \quad \mbox{for} \quad \vartheta >\max\{\vartheta^+_l(s_2), \vartheta^+_l(s_1)\},$$
     then they will be called    {\em  parallel} curves. Moreover, by \eqref{parallelself},   $i_{l,x_0}(\vartheta)$ and $i_{l,x_0}(\vartheta+2\pi)$
will also be called parallel.\end{rem}
\begin{theorem}\label{regi(sigma)}
 If  $d\sigma$ is the  arc element of the involute $i_{l,x_0}$ then $\vartheta \to \sigma(\vartheta)$ is 
 continuous  and invertible  in $\vartheta \geq \vartheta_0^+$ with continuous inverse $[0, +\infty)\ni \sigma \to \vartheta(\sigma)$. Moreover
 \begin{equation}\label{dsigma} 
  d\sigma =\left(s_{l+}(\vartheta)-s_0\right) d \vartheta  \quad  \mbox{for} 
  \quad\vartheta \geq \vartheta_0^+,\, \vartheta\not\in \mathfrak{F};
 \end{equation} 
 the  involute is a convex curve with positive curvature a.e.
\begin{equation}\label{curvatureinv} 
 \frac{d\vartheta}{d \sigma}=\frac1{(s_{l+}(\vartheta)-s_0)}\quad  \mbox{for} \quad\vartheta > \vartheta_0^+,\,
 \vartheta\not\in \mathfrak{F},
\end{equation}
   $\sigma \to i_{l,x_0}(\vartheta(\sigma))$ is $C^1$ everywhere and 
  \begin{equation}\label{di/dsigma} 
  \frac{d}{d\sigma}i_{l,x_0} ={\bf n}_{\vartheta(\sigma)}.
 \end{equation} 
 
Moreover the following properties hold.
\begin{enumerate}
 \item[i)] For every $\sigma >0$ the right derivative
 $$(\frac{d \vartheta}{d \sigma})^+=\frac{1}{s_{l+}(\vartheta(\sigma))-s_0}$$
 exists everywhere and it is a decreasing  cadlag function;
\item[ii)] $\frac{d}{d \sigma}i_{l,x_0}$ has everywhere right  derivative given by
$$(\frac{d^2}{d \sigma^2}i_{l,x_0})^{+}=- \frac{1}{s_{l+}(\vartheta(\sigma))-s_0}{\bf t}_{\vartheta(\sigma)}.$$
\end{enumerate}
 \end{theorem}
\begin{theorem}\label{KntoK}
 Let  $K^{(n)} $ be a sequence of plane convex bodies which converges uniformly
to $K$, $x^{(n)}\in \pa K^{(n)}$, $x^{(n)}\to x_0$; then 
   the corresponding  sequences of 
 left  involutes $ i^{K^{(n)}}_{l,x^{(n)}}$ 
  converge  uniformly  to $ i_{l,x_0}$ in compact subsets  of $[\vartheta_0^+,+\infty]$; moreover  the corresponding sequence of
 their derivatives (with respect to the arc length) converges uniformly to  
 $\frac{d}{d\sigma}i_{l,x_0}$.
\end{theorem}
\begin{proof} By Remark \ref{snconverge} the sequence of functions $s_{l+}^n$ converge to $s_{l+}$.
 From \eqref{dsigma} 
 the arclengths of the left involutes $ i^{K^{(n)}}_{l,x^{(n)}}$ 
 $$\sigma^{(n)}(\vartheta)=\int_{\vartheta_0}^{\vartheta} \left(s_{l+}^{(n)}(\vartheta)-s^{(n)}_0\right) d \vartheta  $$
 converges uniformly in compact substes of $[\vartheta_0^+, +\infty)$ to  the arc length $\sigma(\vartheta)$
 of $i_{l,x_0}$;  from \eqref{di/dsigma} the same fact holds  for  their derivatives.
\end{proof}

\begin{figure}[htb]
\epsfig{file=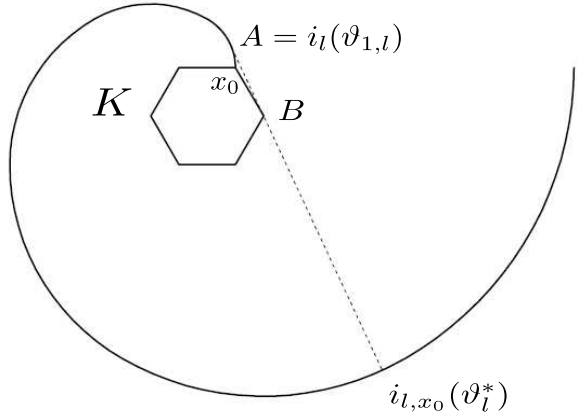, width=10cm} \caption{Left involute of an exagon } \label{exagon}
\end{figure}

Let us consider the  arc of the involute: 
$$\eta:=\{i_{l,x_0}(\vartheta):\vartheta^+_0 \leq \vartheta \leq \vartheta^+_0+3\pi/2 \}$$ 
and the set valued map $F$ (Definition \ref{defGF}).
Let
$$Q=\bigcup_{\vartheta^+_0 \leq \vartheta \leq \vartheta^+_0+3\pi/2 }\{\lambda F(\theta)
+(1-\lambda)i_{l,x_0}(\vartheta), \quad 0\leq \lambda \leq 1 \}, \, \theta=(\cos\vartheta, \sin \vartheta),$$

the union of segments joining the points of $\eta$ with the corresponding points on $\pa K$.
\begin{definition}\label{deftheta^*}
 If the tangent sector  $T(x_0)$ to $K$ has an opening less  or equal than $\pi/2$ as in Fig.\ref{square}, then $Q\cup K$ is convex; 
 let us define 
 $$\vartheta^*_l=\vartheta^+_0+3\pi/2.$$
 If $Q\cup K$ is not convex then let  us consider $co(Q\cup K)$. Let us notice that 
 $\pa co(Q\cup K)\setminus \pa (Q\cup K)$ is an open segment with end points $A,B$, with $A\in \eta$, $B\in \pa K$.
 Let us define $\vartheta^*_l$, with $\vartheta^+_0+3\pi/2\leq  \vartheta^*_l < \vartheta^+_0+2\pi$ 
 such that (see Fig.\ref{exagon})
 $\theta^*_l=(cos \vartheta^*_l, \sin \vartheta^*_l )$ is orthogonal to $AB$,  $B\in F(\theta^*_l)$.
 Let $\vartheta_{1,l}$ be the smallest $\theta > \theta_0^+$  satisfying $A=i_{l,x_0}(\vartheta_{1,l})$.
 Clearly $\vartheta^*_l=\vartheta_{1,l}+\frac32\pi$.

For the right involutes  a value
$\vartheta^*_r $  is defined similarly,
with
$  \vartheta_0^--2\pi <\vartheta^*_r \leq  \vartheta_0^--3\pi/2$, such that the line orthogonal to 
$\theta^*_r $ supporting $K$ at $F(\theta^*_r)$ is tangent to the right involute at
$i_{r,x_0}(\vartheta_{1,r})$, see Fig.\ref{circle} where  $F(\theta^*_r)$ is 
the point $x(s_{r-}(\vartheta^*_r))$, written as  $x(\vartheta^*_r)$ for short.

\end{definition}

 \begin{figure}[htb]
\epsfig{file=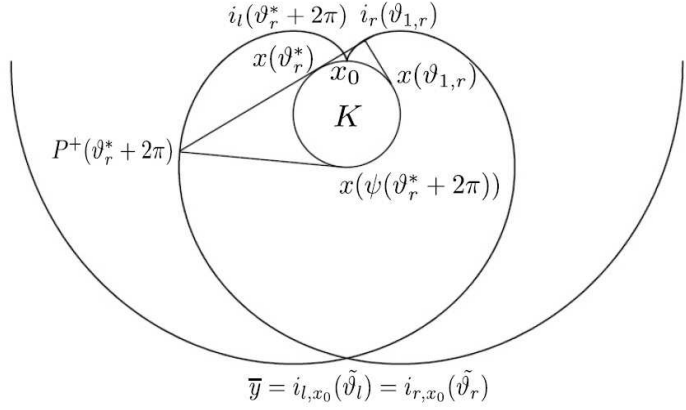, width=10cm} \caption{ Involutes of a circumpherence } \label{circle}
\end{figure}

\begin{theorem}\label{increasing sec} Let  $i_l:=i_{l,x_0}$ be the left involute starting at $x_0$ on the boundary of 
a plane convex body $K$, then:  
\begin{enumerate}
 \item[i)] the left involute  $\vartheta \to i_l(\vartheta)$  
has the  distancing from $K$  property for $\vartheta \geq \vartheta_0^+$,
but is not a SDC for  $\vartheta \geq \vartheta^*_l$;
\item[ii)] the curve $\vartheta \in [\vartheta^+_0, \vartheta^*_l]\to i(\vartheta)$ is 
a SDC; 
 \item[iii)] for  $y\in Int(K)$ the distance function $J_y(\vartheta)=|i_l(\vartheta)-y|$
 is strictly increasing for 
$\vartheta \geq \vartheta^+_0$;
\item[iv)] if $y\in \pa K$, then $J_y(\vartheta)$  is not decreasing for 
$\vartheta \geq \vartheta^+_0$ and $\frac{d}{d\vartheta}J> 0$ for $(\cos \vartheta, \sin \vartheta) \not \in  N_K(y)$.

\end{enumerate}
\end{theorem}
\begin{proof} As $i_l$ is rectifiable, then the function 
  $J^2_y(\vartheta)=|i_l(\vartheta)-y|^2$ is an absolutely continuous function for $\vartheta\geq \vartheta^+_0$, 
  and from \eqref{di(theta)} for $\vartheta\not \in \mathfrak{F}$
  $$\frac12\frac{d}{d\vartheta}J^2_y=\langle \frac{d}{d\vartheta}i_l, i_l(\vartheta)-y\rangle=
  \langle (s_{l}(\vartheta)- s_0){\bf n}_\vartheta, x_l(s_{l+}(\vartheta))+
  (s_{l+}(\vartheta)- s_0){\bf t}_\vartheta-y\rangle=$$
  $$=(s_{l+}(\vartheta)- s_0)\langle {\bf n}_\vartheta, x_l(s_{l+}(\vartheta))-y\rangle \geq 0;$$
the last inequality holds since ${\bf n}_\vartheta$ is the outer normal to $\pa K $ at $x_l(s_{l+}(\vartheta))$. 
Moreover the previous inequality is strict for all $\vartheta$  if $y\in Int(K)$, 
it is also a strict inequality for $y\in \pa K$ and    $y \not \in F(\theta)$. This proves iii) and iv).
Then i) follows from iii) and  Definition \ref{defsecA}  of distancing from K property for a  curve.
To prove ii) let us recall that a SDC satisfies   \eqref{defselexpwithder}; then one has 
to prove that the angle at $i_l(\vartheta)$  between  the vector  
$i_l(\vartheta)-i_l(\tau)$, $\vartheta_0^+< \tau < \vartheta \leq \vartheta^*_l$,
and  ${\bf n}_{\vartheta}$, the tangent vector at $i_l(\vartheta)$,
is greater or equal than $\pi/2$; this is equivalent
to show that the half line $r_\vartheta$ through $i_l(\vartheta)$ and $x(s_{l+}(\vartheta))$
orthogonal to ${\bf n}_{\vartheta}$
supports at $i_l(\vartheta)$ the arc of  $i_l$ from $x_0$ to $i_l(\vartheta)$.
By Definition \ref{deftheta^*} this is the case for all $\vartheta$ between $\vartheta^+_0$ and $\vartheta^*_l$.
\end{proof}

\begin{cor}\label{invasseck} The left involute  $\vartheta \to i_{l,x_0}(\vartheta)$ of the boundary of a plane convex body $K$ 
is a self-distancing  curve from $K$  for $\vartheta \in [\vartheta^+_0, \vartheta^*_l]$; similarly
the right involute \eqref{parinvtheta-} is a self-contracting curve from $K$
for $\vartheta \in [\vartheta^*_r,\vartheta^-_0]$.
\end{cor}
\begin{proof}
From i) of Theorem \ref{increasing sec} the left involute is a curve such that the distance of its points from all
 $y\in K$ is not decreasing;  ii) of the same theorem proves that it is a  SDC. Let us recall that a self-contracting curve is a self-distancing 
 curve with opposite orientation.
\end{proof}

\begin{theorem}\label{thetaintersection}Let $K$ be a plane convex body not reduced to a single point and let 
$x_0,s_o,\theta_0,\vartheta_0$ be the initial parameters.
Let 
$[\vartheta^+_0, \vartheta^+_0+2\pi]\ni  \vartheta \to i_{l}(\vartheta)$ be an arc of the left  involute starting at $x_0$,
 $[\vartheta^-_0-2\pi,\vartheta^-_0]\ni  \vartheta \to i_{r}(\vartheta)$ be an arc of the right  involute ending at  $x_0$;
 then there exists only one point $\overline{y}\neq x_0$ which belongs to both arcs and
 \begin{equation}\label{xsegnato}
  \overline{y}= i_{l}(\tilde{\vartheta_l})=i_{r}(\tilde{\vartheta_r}), 
 \end{equation}
 with  
  $$\vartheta^-_0 \leq \vartheta^*_r+2\pi <\tilde{\vartheta_l}< \vartheta_0^+ +3\pi/2 \leq \vartheta^*_l, $$
  $$\vartheta^*_r \leq \vartheta_0^- -3\pi/2 < \tilde{\vartheta_r}< \vartheta^*_l-2\pi\leq \vartheta^+_0. $$
 
\end{theorem}
\begin{proof} For simplicity, first let us prove the existence of $\overline{y}$ assuming that $K\in C^2_+$.  
With the assumed conditions , $\RR \ni \vartheta  \to x(\vartheta):=x(\theta)$ defined by 
\eqref {supfunctionplane} is a parameterization  of $\pa K$.

  Let $\vartheta \in [\vartheta_0, \vartheta_0+2\pi]$ and let
   $P^+(\vartheta)$ be the first common point of the half line
      $\{x(\vartheta)+\la{\bf t}_\vartheta, \la >0 \}$ and of $i_{l}$.
      Moreover, let  $[\vartheta_0, \vartheta^*_l]\ni \vartheta \to \psi(\vartheta)$ 
      be the function satisfying 
      \begin{equation}\label{p^+^i_l} 
   P^+(\vartheta)=i_l(\psi(\vartheta)).
       \end{equation}
     Let
       $$\phi(\vartheta):= |P^+(\vartheta)-i_l(\vartheta)|.$$
       First the following sentence will be  proved:\\
    {\em    Claim  1}:.
    $P^+(\overline{\vartheta_l})$ belongs to $i_{r}$
 iff  the equality
\begin{equation}\label{phi(thetabar)}   
   \phi(\overline{\vartheta_l})=L
  \end{equation}
 holds for some $\overline{\vartheta_l}\in [\vartheta_0, \vartheta_0+2\pi]$, $L=|\pa K|$.

 {\em Proof of Claim 1}.\\
 If  \eqref{phi(thetabar)}  holds, then
 $$|P^+(\overline{\vartheta_l})-x(\overline{\vartheta_l})|=
 |P^+(\overline{\vartheta_l})-i_{l}(\overline{\vartheta_l})|-
 |i_{l}(\overline{\vartheta_l})-x(\overline{\vartheta_l})|=$$
 $$=
  L-|arc^+(x_0,x(\overline{\vartheta_l}))|=
  |arc^+(x(\overline{\vartheta_l}),x_0)|=|arc^-(x_0,x(\overline{\vartheta_l}-2\pi)  )|.$$
 Thus 
  $$P^+(\overline{\vartheta_l})=
  x(\overline{\vartheta_l})+|arc^+(x(\overline{\vartheta_l}),x_0)|{\bf t}_{\overline{\vartheta}_l}=
  x(\overline{\vartheta_l}-2\pi)+|arc^-(x_0,x(\overline{\vartheta_l}-2\pi))|{\bf t}_{\overline{\vartheta_l}-2\pi}=
  i_{r}(\overline{\vartheta_l}-2\pi).$$
 Thus  $P^+(\overline{\vartheta_l})$ is on both arcs of involutes and the other way around.
 
       Our aim is to prove that there exists 
       $\overline{\vartheta_l}\in [\vartheta_0, \vartheta_0+3\pi/2]$ such that \eqref {phi(thetabar)} holds.
       For this goal we prove next Claim 2 and Claim 3.

{\em Claim 2}: The following facts hold in  $[\vartheta_0, \vartheta^*_l]$:
   \begin{enumerate}
   \item[i)] $\psi$ is continuously differentiable and $\psi'>0$,
   \item[ii)]   $\phi'>0$.
   \end{enumerate}
     {\em Proof of Claim 2.} \\
     Let us prove that ${\bf n}_\vartheta $ and ${\bf n}_{\psi(\vartheta) }$ satisfy
     \begin{equation}\label{anglefranthetanapsi}
      \langle {\bf n}_\vartheta,{\bf n}_{\psi(\vartheta)} \rangle < 0.     
     \end{equation}
Let us consider the triangle with vertices $x(\vartheta), i_l(\psi(\vartheta)), x(\psi(\vartheta))$. As
$$|i_l(\psi(\vartheta))-x(\psi(\vartheta))|=|arc^+(x_0,x(\psi(\vartheta)))| \geq 
|arc^+(x(\vartheta),x(\psi(\vartheta))| \geq |x(\psi(\vartheta))-x(\vartheta)|,$$
the angle between   $x(\psi(\vartheta))-P^+(\vartheta)$ and $x(\vartheta)- P^+(\vartheta)$ is acute  and 
the angle between ${\bf n}_\vartheta $ and ${\bf n}_{\psi(\vartheta) }$ is obtuse.
Thus \eqref{anglefranthetanapsi} follows. By definition, $\psi(\vartheta)$ solves \eqref{p^+^i_l}, 
thus $\psi(\vartheta)$
is the implicit solution to
\begin{equation}\label{dinipsi}
 \langle i_l(\psi(\vartheta))-x(\vartheta), {\bf n}_\vartheta \rangle= 0.
\end{equation}
As
$$\langle \frac{d}{d \psi}i_l(\psi), {\bf n}_\vartheta \rangle =(s(\psi)-s_0)\langle {\bf n}_\psi, {\bf n}_\vartheta \rangle$$
is negative by \eqref{anglefranthetanapsi}, then by Dini's Theorem  
equation \eqref{dinipsi} 
has a solution $\psi(\theta)$ satisfying
$$(s(\psi)-s_0)\langle {\bf n}_\psi,{\bf n}_\vartheta \rangle\psi'(\vartheta)+ 
\langle i_l(\psi(\vartheta))-x(\vartheta), {\bf t}_\vartheta \rangle=0.$$
As $i_l(\psi(\vartheta))-x(\vartheta)=\lambda {\bf t}_\vartheta (\lambda > 0)$ 
and $\eqref{anglefranthetanapsi}$ holds, then  $\psi'>0$, $\psi$ is strictly increasing and continuously differentiable.

Let us prove (ii). The formula
\begin{equation}\label{dphdtheta^2}
 \frac{d}{d \vartheta}|i_l(\vartheta)-i_l(\psi(\vartheta))|^2=
2\langle i_l(\vartheta)-i_l(\psi(\vartheta)),
\frac{d}{d \vartheta}i_l(\vartheta)-\frac{d}{d \vartheta}i_l(\psi(\vartheta))\rangle
\end{equation}
holds. Let us notice that $i_l(\vartheta)-i_l(\psi(\vartheta)$ is parallel to ${\bf t}_\vartheta$; thus by
\eqref{di(theta)} 
$$\langle i_l(\vartheta)-i_l(\psi(\vartheta)), \frac{d}{d \vartheta}i_l(\vartheta)\rangle=0.$$
On the other hand  
$$-\langle i_l(\vartheta)-i_l(\psi(\vartheta)),
\frac{d}{d \vartheta}i_l(\psi(\vartheta))\rangle=-\langle-s(\vartheta){\bf t}_\vartheta-\lambda {\bf t}_\vartheta,
(s(\psi(\vartheta))-s_0){\bf n}_{\psi(\vartheta)}\rangle\psi'$$
$$=(s(\vartheta)+\lambda)(s(\psi(\vartheta))-s_0)\langle {\bf t}_\vartheta,{\bf n}_{\psi(\vartheta)}\rangle \psi'.$$
As the angle between ${\bf t}_\vartheta$ and ${\bf n}_{\psi(\vartheta)}$ is acute, then 
last term in the above equalities 
is positive; thus the derivative in the left hand side of \eqref{dphdtheta^2} is positive
and (ii) of Claim 2 follows.

{\em Claim 3}:
In the interval $[\vartheta_0,\vartheta^*_l]$ the function  $\phi$ 
has values smaller than $L$ and greater than $L$.
  
 {\em Proof of Claim 3}.\\
 The angles  $\vartheta^*_r$, and $\vartheta_{1,r}$ has been introduced in Definition  \ref{deftheta^*}. For simplicity
 $x(s_{r-}(\vartheta_{1,r}))$ will be denoted with $x(\vartheta_{1,r})$.
 Let us consider the convex set bounded by $arc^+(  x(\psi(\vartheta^*_r+2\pi)), x(\vartheta_{1,r}))$ and by 
 the polygonal line with vertices 
    $x(\vartheta_{1,r}),i_r(\vartheta_{1,r}),P^+(\vartheta^*_r+2\pi), x(\psi(\vartheta^*_r+2\pi))$, 
    see Fig.\ref{circle}.
     
 Clearly the inequalities
  $$|i_r(\vartheta_{1,r})-P^+(\vartheta^*_r+2\pi)|< |i_r(\vartheta_{1r})-x(\vartheta_{1,r})|+
  |arc^-(x(\vartheta_{1,r}),x(\psi(\vartheta^*_r+2\pi))|+|x(\psi(\vartheta^*_r+2\pi)-P^+(\vartheta^*_r+2\pi)|=$$
  $$=|arc^-(x_0,x(\vartheta_{1,r}))|+|arc^-(x(\vartheta_{1,r}),x(\psi(\vartheta^*_r+2\pi))|+
  |arc^-(x(\psi(\vartheta^*_r+2\pi),x_0)|=L$$
  hold.
  As 
  $$\phi(\vartheta^*_r+2\pi)=|i_l(\vartheta^*_r+2\pi)-P^+(\vartheta^*_r+2\pi)|
  <|i_r(\vartheta_{1,r})-P^+(\vartheta^*_r+2\pi)|,$$
  using the previous inequalities, one obtains
 $$    \phi(\vartheta_r^*+2\pi) <L.$$
 Let us show now that
\begin{equation}\label{phivartheta0)+3pi/2>L}
   \phi(\vartheta_0+3\pi/2) > L
 \end{equation}
 holds.
 
Let $\mathnormal{\rho}$ be the half line with origin  $x_0$ and direction  
$-{\bf t}_{\vartheta_0}$; $\mathnormal{\rho}-\{x_0\}$
crosses the arc $i_{r}$ in a first point  $y_1=i_r(\alpha_1)$, with $ \alpha_1 < \vartheta_0-\pi/2$.
Then 
    $$r:=|x_0-y_1|< |y_1-x(\alpha_1)|+|arc^-(x(\alpha_1),x_0)|=L.$$
    The half line  $\mathnormal{\rho}$ meets the arc
    $i_{l}$ in a point $y_2$ and
     $|y_2-x_0|=L.$

     Property (iii) of Theorem \ref{increasing sec} implies that  the arc 
     $\mathnormal{D}$ of the left involute after $y_2$ 
     lies outside of the circle centered in $x_0$ and with radius $L$. The similar property for the 
        right involute implies that  the arc $\mathnormal{C}$ 
    of the right involute joining $x_0$ to $y_1$ lies in the circle with center
     $x_0$ and radius $r$; thus the straight line tangent to $K$ at  $x(\vartheta_0+3\pi/2)$  
 meets the arc $\mathnormal{C}$  in $i_r(\vartheta_0-\pi/2)$ and 
 $\mathnormal{D}$ in $P^+(\vartheta_0+3\pi/2)$.
 Therefore 
 $$\phi(\vartheta_0+3/2\pi)=|i_l(\vartheta_0+3\pi/2)-P^+(\vartheta_0+3\pi/2)|=$$
 $$=
 |i_l(\vartheta_0+3\pi/2)-x(\vartheta_0+3\pi/2)|+|x(\vartheta_0+3\pi/2)-P^+(\vartheta_0+3\pi/2)|> $$
$$>|i_l(\vartheta_0+3\pi/2)-x(\vartheta_0+3\pi/2)|+|x(\vartheta_0+3\pi/2)-i_r(\vartheta_0-\pi/2)|= $$
$$=|arc^+(x_0,x(\vartheta_0+3\pi/2)|+|arc^-(x_0,x(\vartheta_0+3\pi/2))|=L.$$
\eqref{phivartheta0)+3pi/2>L} is proved.
  
  The intermediate values  theorem  implies  that there exists 
  $\overline{\vartheta_l}\in [2\pi+\vartheta^*_r, \vartheta_0+3\pi/2]$
  such  that \eqref{phi(thetabar)} holds.  
  Claim 1 implies that
  $$P^+(\overline{\vartheta_l})=i_l(\psi( \overline{\vartheta_l}))=i_r(\overline{\vartheta_l}-2\pi),$$
    so the right involute and the left involute 
  meet each other in one point and \eqref{xsegnato}  is proved with 
  $\tilde{\vartheta_l}=\psi( \overline{\vartheta_l})$, $\tilde{\vartheta_r}=\overline{\vartheta_l}-2\pi$.

  By approximation argument the same result holds for an arbitray convex body K.
   
 Let us prove now that the point $\overline{y}$ is unique. Let us argue by contradiction.
  Let  $P,Q$ be two distinct points on  $i_{l}\cap i_{r}$, with $P \prec Q$ on $i_l$ and $i_r$;
  then since $i_l$ is a distancing curve from $x_0$:
  $$|P-x_0| \leq |Q-x_0|,$$
  and since $i_r$ is a contracting curve to $x_0$:
  $$|P-x_0| \geq |Q-x_0|.$$
  Therefore all the points on the arc of $i_l$ and of $i_r$ between $P$ and $Q$ have the same distance 
  from $x_0$; thus, between $P$ and $Q$, $i_l$ and $i_r$ 
  (arc of involutes of a same convex body $K$) coincide with the same  arc of circle centered at $x_0$, this implies that $K$ reduce 
  to the point $x_0$, which is not possible for the assumption. 
  \end{proof}
     \begin{definition}\label{defzl}
  Let $z \not \in K$. Let $z_l (z_r) \in \pa K $ on the contact set on the ``left'' (right) support line to $K$
   through $z$.
 If the contact set is a 1-face on these support lines, then $z_l$ and $z_r$ are identified as the closest ones to $z$.
 The triangle $zz_lz_r$ is  counterclockwise oriented.
  \end{definition}
  
  \begin{theorem}\label{family of involutes} 
 For every $\xi\in \pa K$ let us consider 
  the left involutes $i_{l,\xi}$ and the right involutes $i_{r,\xi}$
  parameterized by 
  their arc length $\sigma$.
   The maps 
   $$ \pa K\times (0,+\infty)\ni (\xi, \sigma) \to i_{l,\xi}(\theta(\sigma)) \in \RR^2\setminus K, $$
   $$ \pa K\times (0,+\infty)\ni (\xi, \sigma) \to i_{r,\xi}(\theta(\sigma)) \in \RR^2\setminus K $$
  are  1-1 maps. 
  \end{theorem}
\begin{proof} Assume, in the proof, that $x_0\in \pa K, \theta_0\in G(x_0), \vartheta_0, s_0$ are fixed.
Let $z \not \in K$. The tangent sector to the cap body $K^z$ with vertex z has two maximal segments
$zz_l$, $zz_r$ on the sides that do not meet $K$ (except  at the end points  $z_l$,$z_r$).
Let $\vartheta_l $ such that $z_l=x_l(s_{l+}(\vartheta_l))$, and let $\overline{s}$ such that
$$|z-z_l|=s_{l+}(\vartheta_l)-\overline{s}.$$
Let $\xi_l=x_l(\overline{s})$, let $\overline{\vartheta}=\vartheta_l^+(\overline{s})$. From \eqref{remiv} 
 and from the definition of left involute \eqref{defleftinvolutegeneral} (with $\xi_l$ in place of $x_0$, 
$\overline{\vartheta}$ in place of $\vartheta_0^+$, $\overline{s}$ in place $s_0$)
$$z=i_{l,\xi_l}(\vartheta_l)$$
holds; thus the map $(\xi, \sigma) \to i_{l,\xi}(\sigma)$ is surjective. Moreover the map  it is also injective,
since the left involutes don't cross each other since they are parallel (see Remark \ref{parallel}). Similar proof holds 
fir the  right involutes.
\end{proof}
Let $\xi_l=x_l(\overline{s})$ be the starting point of  the  left involute
$i_{l,\xi_l}$  through $z$, defined in the previous theorem; similarly let   $\xi_r$ be the starting point of the
 right involute $i_{r,\xi_r}$ through $z$.
Let us notice that $i_{l,\xi_l}$ and $i_{r,\xi_r}$ meet each other in a countable ordered set of  points.

\subsection{$\mathfrak{J}$-fence and $\mathfrak{G}$-fence}\label{fencesintro}
\begin{definition}\label{jfence} Let $ K $ be a convex body in $\R^2,  |\partial K | >0, $
$x_0 \in \pa K, \theta_0\in G(x_0), \theta_0=(\cos \vartheta_0, \sin \vartheta_0), s_0\in \RR$.  Let $i_l:=i_{l,x_0}, i_r:=i_{r,x_0}$. Let
$$\overline{y}= i_{l}(\tilde{\vartheta_l})=i_{r}(\tilde{\vartheta_r})\in \RR^2\setminus K,$$
be the first point where the two involutes cross each other (see Theorem \ref{thetaintersection}).
Let us define
$$
\mathfrak{J}_l(K, x_0 ) := \{ y \in \R^2 : y = t x_0 +
(1-t)i_{l}(\vartheta), \quad 0\leq t \leq 1,  \vartheta_0^+ \leq \vartheta \leq \tilde{\vartheta_l}\},
$$
$$\mathfrak{J}_r(K, x_0 ) := \{ y \in \R^2 : y = t x_0 +
(1-t)i_{r}(\vartheta), \quad 0\leq t \leq 1,   \tilde{\vartheta_r}\leq  \vartheta \leq  \vartheta_0^- \},$$
$$\mathfrak{J}(K, x_0 ):=(\mathfrak{J}_l (K, x_0 ) \cup \mathfrak{J}_r(K, x_0 ))\setminus Int(K) .$$
$\mathfrak{J}(K, x_0 )$ will be called the $\mathfrak{J}$-fence of $K$ at $x_0$.
\end{definition}
Let us notice that  $\mathfrak{J}_l (K, x_0 )$ and  $\mathfrak{J}_r(K, x_0 )$ are 
two convex bodies with in common the segment
$x_0 \overline{y}$ only.

From Theorem \ref{family of involutes} the starting point $\xi_l$ ($\xi_r$) 
of a left(right) involute is uniquely determined from  any point $z\not \in K$   of  the involute.
The arc of the points on the left (right) involute between the starting point and $z$ will be denoted by
$i_{l,\xi_l}^z $ ($i_{r,\xi_r}^z$ ), or $i_{l}^z$ ($i_{r}^z$) for short.
 For $y\preceq w$ let us denote with $i_l^{y,w} (i_
r^{y,w})$ the oriented arc of the left (right) involute between $y$ and $w$. 

Let us introduce now other regions which are bounded by left and right involutes. 

Let us fix the initial parameters $x_0,s_0,\theta_0,\vartheta_0$.
 
\begin{definition}\label{defdefT_l(K,z)} Given $z\in \RR^2\setminus K$, let $i_l=i_{l,\xi_l}$ ($i_r=i_{r,\xi_r}$) be
the left (right)
involute through $z$ with starting point
$\xi_l$ ($\xi_r$) and let $z_l(z_r)\in \pa K$ be as in Definition \ref{defzl}.  
Let  $\vartheta_{\xi_l}^+$ satisfying $x_l(s_{l+}(\vartheta_{\xi_l}^+))=\xi_l$. 
Let  $\vartheta_l > \vartheta_{\xi_l}^+$ be the smallest angle for which $x_l(s_{l+}(\vartheta_l))=z_l$.
Let us consider 
the parameterization \eqref{defleftinvolutegeneral}; let us define
\begin{equation}\label{defT_l(K,z)}
 \mathfrak{G}_l(K, z ):=\{tx_l(s_{l+}(\vartheta))+(1-t)i_l(\vartheta), \quad 0 < t < 1, \quad 
\vartheta_{\xi_l}^+<  \vartheta < \vartheta_l \}.
\end{equation}
If $i_l^{z}$ does not cross the open segment $zz_l$,  the  region $\mathfrak{G}_l(K, z )$  is an
open set bounded by the convex arc of left involute $i_l^{z}$, 
the segment  $zz_l$ and the convex arc of $\pa K$: $arc^+(\xi_l,z_l)$; otherwise let $w$ be the nearest point to $z$ where 
 $i_l^{z}$ crosses the open  segment $zz_l$; the region $\mathfrak{G}_l(K, z )$  is an
open set bounded by the  arc  $i_l^{w,z}$, the segment $wz$ and $\pa K$. Similarly let us define  $\mathfrak{G}_r(K, z )$.
\end{definition}
$\mathfrak{G}_l(K, z ), \mathfrak{G}_r(K, z )$ are  open and  bounded sets. 
Let us define:
\begin{equation}\label{defmathfrakG(K, z)}
\mathfrak{G}(K, z ):=Int(cl(\mathfrak{G}_l(K, z )\cup\mathfrak{G}_r(K, z ))).
\end{equation}
$\mathfrak{G}(K, z )$ is an open, bounded, connected set. 
$\mathfrak{G}(K, z )$ will be called the $\mathfrak{G}$-fence of $K$ at $z$.

\begin{rem} If $z$ is the first crossing point of $i_l$ and $i_r$ and $\xi_l=\xi_r$,  then
 $\mathfrak{G}(K,z)=Int (\mathfrak{J}(K,\xi_l))$. 
\end{rem}
  Let us conclude this section with the following result, which follows from Theorem \ref{KntoK}.
\begin{theorem}\label{stability j-fences} 
Let  $K$ be limit of a sequence of  convex bodies $K^{(n)}$,
 $x_0=\lim x_0^{(n)}$, $ x_0^{(n)}\in \pa K^{(n)}$. Then
$$\mathfrak{J}(K^{(n)},x_0^{(n)})\to \mathfrak{J}(K,x_0).$$
Moreover if $z\not \in K, z=\lim z^{(n)}, z^{(n)}\not \in K^{(n)}$, 
then
$$cl(\mathfrak{G}(K^{(n)},z^{(n)})) \to cl(\mathfrak{G}(K,z)).$$
 \end{theorem}

\section{ Bounding regions for SDC in the plane}\label{fences}
Let us assume that $x_0$ is the end point of one of the following sets
\begin{enumerate}
 \item[a)] a steepest descent curve  $\ga$, satisfying  \eqref{defselexpwithder} and \eqref{expandingproperty};
 \item[b)] $\ga^K$: a self-distancing curve from a convex body $K$, see Definition \ref{defSEC+}.
\end{enumerate}
The following questions arise: can one extend   $\ga$, $\ga^K$ beyond $x_0$?
Which regions delimit that extension? Which regions are allowed and which are forbidden?

\begin{lemma}\label{leftinvincludedinleft}
 Let $z\in \RR^2\setminus K$. 
  If $u\in \mathfrak{G}_l(K, z )$
 then  the arc $i_{l}^{u}$  of the left involute to $K$ ending at $u$,  is contained  in
 $\mathfrak{G}_l(K, z )$. Similarly if $u\in \mathfrak{G}_r(K, z )$, 
  then $i_r^{u}\subset \mathfrak{G}_r(K, z )$.
\end{lemma}
\begin{proof} Since $u\in \mathfrak{G}_l(K, z )$,
 by  \eqref{defT_l(K,z)} there exist $\overline{\vartheta_l}\in (\vartheta_{\xi_l}^+ ,\vartheta_l )$, $\tau \in (0,1)$ such that
 $$u=\tau x_l(s_{l+}(\overline{\vartheta_l}))+(1-\tau)i_l(\overline{\vartheta_l}).$$
 Then the arc $i_l^{u}$ is parallel to an arc of the left involute $i_l$ (through $z$) for 
 $\vartheta \in (\vartheta_{\xi_l}^+ ,\overline{\vartheta_l} )$. Then any 
 left tangent segment to $K$ from
  a point of $i_l^{u}$ is contained in the left tangent segment  from  the corresponding point 
 of
  $i_{l}^{z}$.  \end{proof}
\begin{lemma}\label{rightinvincludedinleft}
 Let $z\in \RR^2\setminus K$ and let $u\in \mathfrak{G}_l(K, z )$. 
 There are two possible cases:
 \begin{enumerate}
  \item[i)] if the right involute ending at $u$ does not cross 
  the tangent segment $z_lz$ or it crosses $z_lz$ at a point $q\in \mathfrak{G}_l(K, z )$, then in both cases
  $i_r^{u}\subset \mathfrak{G}_l(K, z )$;
  \item[ii)] if the right involute ending at $u$ crosses 
  the tangent segment $z_lz$ at a  point $q\in z_lz\cap \pa \mathfrak{G}_l(K, z )$, then
    $i_r^{q,u}\setminus \{q\}\subset \mathfrak{G}_l(K, z )$.
 \end{enumerate}
\end{lemma} 
\begin{proof} Since the starting point $\xi_r(u)$ of the right involute ending at $u$ is on $\pa K$,
the distance from $\xi_r(u)$ to  a point of  the left involute $i_l^z$ is not decreasing, see iv) of Theorem \ref{increasing sec};
similarly
 the distance from $\xi_r(u)$ to  a  point of  $i_r^{u}$ is not decreasing. In the case i) the arc $i_r^{u}$ has
its end points in $\mathfrak{G}_l(K, z )$
and by the above distance property it can not cross two times the left involute, then it can not cross the boundary of 
$\mathfrak{G}_l(K, z )$, therefore $i_r^{u}\subset \mathfrak{G}_l(K, z )$; similarly in the case ii) 
 the arc $i_r^{q,u}$ can not cross the boundary of $\mathfrak{G}_l(K, z )$ at most than in $q$; therefore  
all the points of this  arc, except than $q$,  belong to $\mathfrak{G}_l(K, z )$.
\end{proof}

From the previous lemma it follows that
\begin{theorem}\label{lefttrapinlefttrap} Let $z \not \in K$. The following inclusions  hold:
\begin{enumerate}
 \item[a)]if  $u\in \mathfrak{G}_l(K,z) $, then
\begin{equation}\label{cllefttrapinlefttrap}
 cl(\mathfrak{G}_l(K,u)) \setminus \pa K \subset \mathfrak{G}_l(K,z);
 \end{equation}
 \item[b)] if  $u\in \mathfrak{G}_r(K,z) $, then
\begin{equation}\label{clrighttrapinrighttrap}
 cl(\mathfrak{G}_r(K,u)) \setminus \pa K \subset \mathfrak{G}_r(K,z);
 \end{equation}
 \item[c)] if $u \in \mathfrak{G}(K,z)$,  then
 \begin{equation}
  \label{cltrapinclusion}
   cl(\mathfrak{G}(K,u))\setminus \pa K  \subset \mathfrak{G}(K,z).
 \end{equation}
\end{enumerate}
 \end{theorem}
 \begin{proof}
  By Lemma \ref{leftinvincludedinleft} the left involute that bounds $\mathfrak{G}_l(K,u)$ is inside $\mathfrak{G}_l(K,z))$,
  then \eqref{cllefttrapinlefttrap} is proved. Inclusion \eqref{clrighttrapinrighttrap} is proved similarly. 
  Let $u \in \mathfrak{G}(K,z)=Int(cl(\mathfrak{G}_l(K, z)\cup\mathfrak{G}_r(K, z )))$ and let us consider 
  $u \in \mathfrak{G}_l(K, z)$, then 
  in case i) of Lemma \ref{rightinvincludedinleft}  also the open arc of the right involute $i_r^{u}$ is inside 
  $\mathfrak{G}_l(K,z))\subset \mathfrak{G}(K,z))$. Besides $i_l^u \subset \pa \mathfrak{G}_l(K,u) $, then
  \eqref{cltrapinclusion} is trivial.  
  In case ii) of Lemma \ref{rightinvincludedinleft} 
  the open arc $i_r^{q,u}$  is inside $\mathfrak{G}_l(K,z))$. On the other hand $q$
  is inside $\mathfrak{G}_r(K,z))$
  and by \eqref{clrighttrapinrighttrap} the arc $i_r^q\subset i_r^u$ is in $\mathfrak{G}_r(K,z))\subset
  \mathfrak{G}(K,z))$. Similar arguments  hold if $u\in  \mathfrak{G}_r(K,z))$.
  Then
  \eqref{cltrapinclusion} holds in this case too. 
 \end{proof}

 \begin{lemma} \label{*polygonal}Let $w\not \in K$. Let $\eta$ be polygonal deleted $SDC_K$ 
  with end point 
 $y\in  \mathfrak{G}(K,w)$.
Then 
 \begin{equation}\label{eta_wintrapG}
  \eta \subset \mathfrak{G}(K,w),
 \end{equation}
 and
 \begin{equation}\label{eta_wincltrapG}
  \eta \subset cl(\mathfrak{G}(K,y)).
 \end{equation}
\end{lemma}
\begin{proof} To prove \eqref{eta_wintrapG}, let us 
 assume, by contradiction, that $\eta$ has a point $z\not \in \mathfrak{G}(K,w)$. With no loss of generality 
 it can
 be assumed that $z \in \pa \mathfrak{G}(K,w)$ and 
 $$\eta\setminus \eta_z \subset \mathfrak{G}(K,w).$$ 
 Then, $z$ is 
  the end point of a segment $zw_i$, where 
 $w_i\in \mathfrak{G}(K,w)\cap \eta$  and $z\prec w_i$ on $\eta$.
 As $z\in \pa \mathfrak{G}(K,w)$, then there exists an involute 
 through $z$ which is a piece of the boundary of $\mathfrak{G}(K,w)$ 
 (to fix the ideas it is assumes that it is the left involute
 $i_l$).
 Let us consider $z_l\in \pa K$ so that the tangent vector ${\bf t}_z$ to $i_l$ at $z$ satisfies
 $$\langle {\bf t}_z, z-z_l\rangle =0.$$
 As $w_i$ is inside the orthogonal  angle centered in $z$ with sides ${\bf t}_z$ and $z_l-z$, then
 $$\langle w_i-z, z-z_l\rangle < 0.$$
 Then as for $\varepsilon > 0 $ sufficiently small,
 $z_\varepsilon:=z+\varepsilon (w_i -z) \in \eta_{w_i}$ and at $z_\varepsilon$  the curve $\eta$ has tangent vector 
 $w_i-z$ that satisfies
 $$\langle w_i-z, z_\varepsilon-z_l\rangle < 0,$$
 contradicting the fact that $\eta_w$ has the distancing from $K$ property \eqref{defselexpwithderfromK}. 
 This proves \eqref{eta_wintrapG}. 
 
 If $w_n \to y$, with $y\in \mathfrak{G}(K,w_n)$, also the inclusions
 $$\eta \subset cl(\mathfrak{G}(K,w_n))$$
 hold. Then \eqref{eta_wincltrapG} is obtained by the approximation Theorem \ref{stability j-fences}. 
\end{proof}

\begin{theorem}\label{*}
 Let $K$ be  a  convex body and let $\ga^K$ be $SDC_K$,  $w\in \ga, w\not \in K$. Then
 \begin{equation}\label{incJKz}
  \ga^K_w \subset cl(\mathfrak{G}(K,w)). 
 \end{equation}
\end{theorem}
\begin{proof} Let us choose a sequence $\{w_n\}, w_n\in \ga^K, w_n \preceq w, w_n \to w $.
Let us fix the arc  $\ga^K_{w_n}$. By \cite[Theorem 6.16]{MLV},  $\ga^K_{w_n}$ is limit of  $SDC_K$ polygonals with
end point $w_n$.
From Lemma \ref{*polygonal}, these polygonals are enclosed in $cl(\mathfrak{G}(K,w_n))$; then 
$$\ga^K_{w_n} \subset cl(\mathfrak{G}(K,w_n))$$
holds too.
The inclusion \eqref{incJKz} is now
obtained from the limit of the previous inclusions and by the approximation Theorem \ref{stability j-fences}.
\end{proof}
 \begin{theorem}\label{bounding regionsJ(K.x_0)} Let $K$ be a convex body not reduced to a point. 
 If $\ga^K$ is a self-distancing curve from $K$ with 
starting point $x_0\in \pa K$, then
\begin{equation}\label{gaouterJ} 
   \ga^K \subset cl(\RR^2\setminus (\mathfrak{J}(K,x_0)\cup K)).
\end{equation}
\end{theorem}
 \begin{proof}Let $z$ be the first crossing point of the left and right involutes of $K$ starting at $x_0$. Then
$$Int(\mathfrak{J}(K,x_0))=\mathfrak{G}(K,z).$$
 By contradiction, if $\ga^K$ has a point $w\in \mathfrak{G}(K,z)$, then, by Theorem \ref{*},  the following inclusion holds
$$\ga^K_w \subset cl(\mathfrak{G}(K,w));$$
since, by the distancing from $K$ property, $\ga^K$ has in common with $K$  only the starting point $x_0$ then, 
the following inclusion 
$$\ga^K_w \setminus\{x_0\}\subset cl(\mathfrak{G}(K,w))\setminus \pa K$$
holds too.
Moreover by \eqref{cltrapinclusion} the set $cl(\mathfrak{G}(K,w))\setminus \pa K$ has positive distance from the $\RR^2\setminus 
\mathfrak{G}(K,z)$; then $\ga^K_w \setminus\{x_0\}$ has a positive distance from 
$\RR^2\setminus\mathfrak{G}(K,z)=\RR^2\setminus Int(\mathfrak{J}(K,x_0))$.
This is in contradiction with
$x_0\in \pa \mathfrak{J}(K,x_0)$.\end{proof}

\begin{cor}\label{corjcogax_0}
 Let $\ga$ be a $SDC$ and $z_1\in \ga$ then
 $$\ga\setminus \ga_ {z_1} \subset cl(\RR^2\setminus \mathfrak{J}(co(\ga_{z_1}), z_1)).$$
\end{cor}
\begin{proof} Since $\ga\setminus \ga_{z_1}$ is a self-distancing curve from $co(\ga_{z_1})$ and $z_1\in \pa co(\ga_{z_1})$ (see 
\cite[(i) of Lemma 4.6]{MLV}, then
 Theorem \ref{bounding regionsJ(K.x_0)} applies to $\ga^K=\ga\setminus \ga_ {z_ 1}$ with $K=co(\ga_{z_1})$.
\end{proof}

 \begin{definition} Let $\ga$ be a SCD.
 If  $z_1, z \in \ga $, with $z_1 \preceq z$ let
 $$\ga_{z_1,z}:=\ga_{z}\setminus \ga_{z_1}.$$
 \end{definition}
 For $z \not \in K$, let $K^z$  be the cap body, introduced in  \eqref{capbody}.
 Next theorem shows the principal result  on bounding regions  for arcs of a SDC $\ga$. 
\begin{theorem}\label{ga-ga_zincludedinJ(capbody)}
  Let $K$ be  a  convex body and let $\ga$ be  a $SDC_K$. If $z_1, z \in \ga $, with $z_1 \preceq z$
then \begin{equation}\label{princboundingregions}
      \ga_{z_1,z}\subset cl(\mathfrak{G}(K,z)\setminus \mathfrak{J}(K^{z_1},z_1)). 
     \end{equation}
 \end{theorem}
\begin{proof} First let us notice that $\ga_{z_1,z}$ has the distancing from $K$ and from the set point $\{z_1\}$ property,
thus by Proposition \ref{lemmaecissec} it has the distancing from $K^{z_1}$ property. Then  the inclusion \eqref{princboundingregions}
follows from 
 Theorems \ref{*} and \ref{bounding regionsJ(K.x_0)}.
\end{proof}


Let us conclude the section with the following inclusion result for $\mathfrak{J}$-fences.

 \begin{theorem}\label{incj}
 Let $K,H$ be two convex bodies not reduced to a point, $K \subset H $. Let $x_0\in \pa K \cap \pa H$. Then
 $$\mathfrak{J}(K,x_0)\subset \mathfrak{J}(H,x_0).$$
\end{theorem}
\begin{proof}
 The boundary of $\mathfrak{J}(H,x_0)$ consists of two arcs of the left and right  involutes of $H$ starting at $x_0$.
By Corollary \ref{invasseck} they are $SDC_H$, then they are $SDC_K$; therefore by Theorem   \ref{bounding regionsJ(K.x_0)}
they cannot intersect the boundary of $\mathfrak{J}(K,x_0)$.
\end{proof}

\subsection{Miniminally connecting plane steepest descent curves }
\begin{figure}[htb]
\epsfig{file=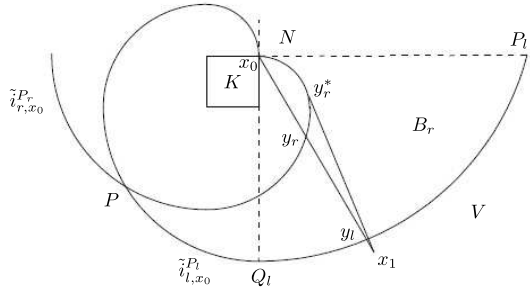, width=10cm} \caption{The regions $N$, ${B}_r$, $V$  when $K$ is a square.} \label{ultimissima}
\end{figure}

Given a point $x_1 \not \in K$, the segment joining it with its projection $x_0$ on $\pa K$ is a $SDC_K$ which {\em 
minimally }
connects the two points.

This subsection is devoted to consider when it would be  possible to connect a given point $x_0$
on the boundary of a plane convex
body $K$, with an  arbitrarily given  point $x_1\not \in K$,  by using a steepest descent 
curve $\ga\in SDC_K$.
Let us denote with $\Ga^K_{x_0,x_1}$ the class of the curves $\ga\in SDC_K$ starting at $x_0$ and ending at $x_1$.

 \begin{definition}
 Let  $\ga$ be a SDC  with end point  $y$ and $\eta$ be a SDC with starting point
 $y$; let us denote by $\ga\star\eta$, the curve joining $\ga$ with $\eta$ in the natural order, if it is a SDC curve.
  \end{definition}
 \begin{theorem}\label{connectingsdc}
  Let $x_0\in \pa K, x_1\not \in K$. Then $\Ga^K_{x_0,x_1}\neq \emptyset$ iff 
  \begin{equation}\label{neccessaryx1outerJ}
   x_1 \in cl(\RR^2\setminus (\mathfrak{J}(K,x_0))\cup K).
  \end{equation}
 If \eqref{neccessaryx1outerJ} holds,  there exist at most two   $\eta_i \in \Ga^K_{x_0,x_1}, i=1,2$ such that the 
 following   properties are true
   \begin{equation}\label{coeta1,2subsetcoga}
    \forall \ga \in \Ga^K_{x_0,x_1}\Rightarrow     
    co(\eta_1) \subset co (\ga) \mbox{\quad or \quad} co(\eta_2)\subset co(\ga)\mbox{\quad (or both)}
   \end{equation}
   and
  \begin{equation}\label{minimallength}
  \forall \ga \in \Ga^K_{x_0,x_1}\Rightarrow  |\ga| \geq \min_{i=1,2}\{|\eta_i|\}.
  \end{equation}
 \end{theorem}
 \begin{proof}Let $\ga \in \Ga^K_{x_0,x_1}$.
  From 
   \eqref{gaouterJ} of Theorem \ref{bounding regionsJ(K.x_0)}, since $x_1 \in \ga$, then  
   \eqref{neccessaryx1outerJ}  follows.
   
   Let us prove now that \eqref{neccessaryx1outerJ} is sufficient. 
   Let us notice that $\RR^2\setminus (\mathfrak{J}(K,x_0)\cup K)$ can be divided 
   in four regions $N, {B}_l,{B}_r, V$, see Fig. \ref{ultimissima}, defined as follows 
   \begin{enumerate}
 \item[i)] the closed normal sector $N:=x_0+N_{K}(x_0)$ is the   angle bounded by the two  half lines 
  $t_l,t_r$ tangent at $x_0\in \pa K$ to the left and right
 involute $i_l:=i_{l,x_0}, i_r:=i_{r,x_0}$  respectively; this angle can be reduced to an half line, starting at $x_0$;
 \item[ii)] let $P$ be the first crossing  point between  $i_l$ and $ i_r$, see Theorem \ref{thetaintersection}; 
  $i_l$ is a SDC  until to $i_l(\vartheta^*_l)$, which will be  a point $Q_l$
 following $P$;  after $Q_l$ the involute $i_l$ is no more a SDC, see i) of Theorem \ref{increasing sec}.

 Let us  change $i_l$ after $Q_l$ with $j_{l,Q_l}$, the  left involute of $co(K\cup i_l^{Q_l})$ at $Q_l$.
  Let us define ${P}_l$ be the first intersection point of $j_{l,Q_l}$ with $\pa N$, and let 
  $$\tilde{i}_{l,x_0}^{{P}_l}:={i}_{l,x_0}^{{Q}_l}\star j_{l,Q_l}^{{P}_l}.$$
 It is not diffcult to see that $\tilde{i}_{l,x_0}^{{P}_l}\in \Ga^K_{x_0,P_l}$.
 Changing the left with the right,  $\tilde{i}_{r,x_0}$ and the point ${P}_r$ 
   can be constructed.  
 Let ${B}_r$ the union of the arc $i_{r,x_0}^P\setminus \{P\}$ with the 
 plane open region bounded by the segment $x_0{P}_l$, the arc $i_{r,x_0}^P$ and
 the arc $\tilde{i}_{l}^{P,{P}_l}$; let ${B}_l$ the union of the arc $i_{l,x_0}^P\setminus \{P\}$ with the 
 plane open region bounded by the segment $x_0{P}_r$,
 the arc $i_{l,x_0}^{P}$ and
 the arc $\tilde{i}_r^{P,{P}_r}$;
 \item[iii)] let $V$ the remaining region, i.e. $V=\RR^2\setminus (K\cup\mathfrak{J}(K,x_0)\cup N
 \cup {B}_l\cup {B}_r)$.
 \end{enumerate} 
   
  Let $x_1\in {B}_r\cup V$. On the oriented curve $\tilde{i}_{r,x_0}^{{P}_r}$
   there are two points so that their  tangent lines contain $x_1$. Let  $y^*_r$ be the first tangency point.
 
 Let $y_l, y_r$ be the intersection points of the half line $m$ starting at $x_0$ and containing $x_1$, with 
  $\tilde{i}_{l,x_0}^{{P}_l}$ and with $\tilde{i}_{r,x_0}^{{P}_r}$ respectively, see Fig. \ref{ultimissima}.
 
  Under the assumption \eqref{neccessaryx1outerJ}, $x_1$ belongs to one of the four regions $N, {B}_l,{B}_r, V$;
  let us prove now \eqref{coeta1,2subsetcoga}, \eqref{minimallength} in the four corresponding cases.
   \begin{enumerate}
  \item If $x_1 \in N$,
 then let  $\eta_1=\eta_2\in \Ga^K_{x_0,x_1}$  be the segment $x_0x_1$.
 Then \eqref{coeta1,2subsetcoga}, \eqref{minimallength}
 are trivial.
 
 \item  Let $x_1\in V$. Let
   \begin{equation}\label{eta_r}
    \eta_r:= \tilde{i}_{r,x_0}^{y^*_r}\star y^*_rx_1.
   \end{equation}
The curve  $\eta_r$ is a $SDC_K$, since the normal lines at all the points on the segment $y^*_rx_1$ have the same 
directions  and support
$\tilde{i}_{r,x_0}^{y^*_r}$ up to 
 $y^*_r$; then  $\eta_r$ is a $SDC_K$ and joins $x_0$ with $x_1$. Similarly is defined the $SDC_K$ 
 \begin{equation}\label{eta_l}
   \eta_l= \tilde{i}_{l,x_0}^{y^*_l}\star y^*_lx_1.
 \end{equation}
 Thus 
 $ \Ga^K_{x_0,x_1}$ is not empty and contains at least the two elements $\eta_l, \eta_r$. 
 
  Let us consider the connected closed curve 
 $$c_{x_1}:=\tilde{i}_{r,x_0}^{y_r}\cup y_ry_l\cup\tilde{i}_{l,x_0}^{y_l}. $$
 
 Let $\ga \in \Ga^K_{x_0,x_1}$, let $T\ni t \to x(t)\in\ga$ be a continuous parameterization of 
  $\ga$. Let us project from $x_0$ the curve  $\ga$ on $c_{x_1}$ and let $D$ be this projection.
  That is, for $ t\in T$, let $\la_t:=\{x_0+\la x(t), 0\leq \la \} $ and let
  \begin{equation}\label{C,D4} 
    D=\cup_{t\in T}\left(c_{x_1}\cap \la_t\right).   
  \end{equation}
  Clearly $D$ is a closed connected subset of $c_{x_1}$ containing $x_0$ and  the segment $y_ry_l$.
Thus $D$ contains
 at least one of the two connected components of $c_{x_1}$ joining $x_0$ with the $y_r,y_l$. Therefore the inclusions
 \begin{equation}\label{casersubsetD}
  (\tilde{i}_{r,x_0}^{y_r}\cup y_ry_l) \subset D,
 \end{equation}
or
\begin{equation}\label{caselsubsetD}
 (\tilde{i}_{l,x_0}^{y_l}\cup y_ry_l) \subset D,
\end{equation}
(or both) hold.

 Assume that \eqref{casersubsetD} holds and let  $\eta_r$ be defined as in  \eqref{eta_r}.
 Since, by construction of $D$, the set $co(D\cup \{x_1\}) $ is contained in 
  $ co(\ga)$, 
  then 
 $$co(\eta_r) \subset co(\ga).$$
  Similarly if  \eqref{caselsubsetD} holds, then 
 $$co(\eta_l) \subset co(\ga),$$
  with $\eta_l$ define by \eqref{eta_l}.
 Then \eqref{coeta1,2subsetcoga} is proved. It is not difficult to see that the region bounded by $\ga\cup x_0x_1$ 
 contains the convex region bounded by $i_{r,x_0}^{y^*_r}\cup y^*_rx_1\cup x_0x_1$. Thus 
 the bound
 $$|\ga| \geq |i_{r,x_0}^{y^*_r}|+|y^*_rx_1|$$
 holds;  similary procedure can be used for  left case.
 This proves \eqref{minimallength}.
   \item Let $x_1\in {B}_r$; the same argument as in the case 2. can be carried on up to  the  inclusions \eqref{casersubsetD},
   \eqref{caselsubsetD}.
   As in the step 2., when  case \eqref{casersubsetD}  holds, 
   the curve $\eta_r$ can be constructed and  $\eta_r$  is a $SDC_K$. 
   
Let us show that if $x_1\in {B}_r$ then 
    \eqref{caselsubsetD} cannot occur, so the curve $\eta_l$ can not to be constructed.
   
   Let us argue by contradiction. If \eqref{caselsubsetD} occurs, then let $z_1\neq x_0$
be the first point where $\ga$ crosses  the half line  $x_0P$. The point $z_1$ exists, since under 
the assumption \eqref{caselsubsetD}, $P\in D$. 
  Then  
  (Theorem \ref{bounding regionsJ(K.x_0)}) 
  $z_1$ does not belong to the open segment $x_0P$. Then, from \eqref{caselsubsetD} (see Definition \ref{jfence})  
  \begin{equation}\label{notgadasinistra}   
  co(\ga_{z_1})\supset \mathfrak{J}_l(K,x_0).
  \end{equation}
 Moreover $\ga\setminus \ga_{z_1}\cup \{z_1\}$ is a $SDC_{co(K\cup \ga_{z_1})}$,
 see Remark \ref{sdcofprevious}.
 Then by \eqref{notgadasinistra}  it is a
  $SDC_{co(K\cup\mathfrak{J}_l(K,x_0))}$.  Let us consider the convex  body $H=co(K\cup i_{l,x_0}^P)$. Since
 $$co(K\cup  \ga_{z_1})\supset co(K\cup\mathfrak{J}_l(K,x_0)) \supset H,$$
 then 
 $\ga\setminus \ga_{z_1}$ is a deleted $SDC_H$.
 From Theorem \ref{ga-ga_zincludedinJ(capbody)}, with $H$ in place of $K$, $x_1$ in place of $z$,
 it follows that
 $$\ga_{z_1,x_1}\subset cl(\RR^2\setminus \mathfrak{J}(H^{z_1},z_1)).$$
 
 Let $P_1$ where the right tangent from $z_1$ to $H$ crosses the arc 
 $\tilde{i}_{l,P}^{P_l}$.  Let us notice that $\tilde{i}_{l,P}^{P_l}$ is also an arc of the left involute of $H$ at $P$.
 Moreover $\tilde{i}_{l}^{P_1,{P}_l}$ is an arc of the left involute of $H^{z_1}$ starting at $P_1$, then it is  
  parallel  to the left involute of $H^{z_1}$ at 
 $z_1$ until it crosses the sector $N$; it turns out 
 that
 $${B}_r\subset \mathfrak{J}(H^{z_1},z_1)).$$
  From the two previous inclusions a contradiction comes out since
  $$x_1\in \ga_{z_1,x_1}\cap B_r =\emptyset.$$
 Then \eqref{caselsubsetD} cannot occur.
\item The case  $x_1\in {B}_l$ is similar to the previous one. 
 \end{enumerate}    The proof is  complete.\end{proof}
    
  \begin{definition}\label{minimalE} Under the  assumptions of Theorem \ref{connectingsdc},
   let us define $E^K_{x_0,x_1}$ the set of $\eta_i,i=1,2$ (possibly coinciding)
   as they are constructed in the proof of Theorem
   \ref{connectingsdc}, which satisfy  \eqref{coeta1,2subsetcoga},\eqref{minimallength}.
   These curves will be  called {\em minimally connecting} steepest descent curves  for the class $\Ga^K_{x_0,x_1}$. 
  \end{definition}
 \begin{definition}\label{defH} 
   Let $\ga:T\ni t \to x(t) $ be an absolutely continuous curve 
   and let $x(t)$ be a point of $\ga$, with tangent vector $\dot{x}(t)$.
  Let
  $$\mathcal{H}_{x(t)}:=\{y\in \RR^2: \langle \dot{x}(t),y-x(t)\rangle \leq 0\}.$$
  $\mathcal{H}_{x(t)}$  is an half plane bounded by the normal line  to $\ga$ at $x(t)$ and it is defined almost everywhere in $T$.
  For the curve  $\ga\setminus \ga_{x_1}$ (consisting of the points of $\ga$ 
  following $x_1$) let us define the  region:
  $$\mathfrak{H}(\ga, x_1):=\bigcap_{x_1 \preceq x(t), x(t) \in \ga } \mathcal{H}_{x(t)}.$$
   If $\mathfrak{H}(\ga, x_1)\neq \emptyset$, then it is a convex set.
   \end{definition}
  If $\ga$ is a SDC ($\ga$ is a $SDC_K$), then  condition \eqref{defselexpwithder} (respectively \eqref{defselexpwithderfromK}) implies that
  \begin{equation}\label{HcontainsgaK}
    \ga_{x_1} \subset \mathfrak{H}(\ga, x_1) \quad 
    ( \mbox{ respectively \quad}    \ga_{x_1} \cup K \subset \mathfrak{H}(\ga, x_1)).
  \end{equation}

  \begin{theorem}\label{connectingwithconditionH} Let $x_0\in \pa K, x_1\not \in K$. Let $\ga_1$ be a  $SDC$ with first point $x_1$. Necessary and sufficient
 conditions for the existence of a curve $\ga$, self-distancing curve from K, starting at $x_0$ and 
 satisfying $(\ga \setminus \ga_{x_1})\cup \{x_1\}=\ga_1$, are as follows:
 \begin{enumerate}
  \item[(a)] $x_1 \in cl(\RR^2\setminus \mathfrak{J}(K,x_0)$);
  \item[(b)] there exists $\eta \in E^K_{x_0,x_1}$ such that $ (K\cup \eta) \subset \mathfrak{H}(\ga_1, x_1)$;
  \end{enumerate} 
   moreover if (a) and (b) are satisfied, then $\ga_1 \in SDC_{co(K\cup\eta)}$.
 \end{theorem}
\begin{proof}
 (a) is necessary by Theorem \ref{bounding regionsJ(K.x_0)}. (b) is necessary by Theorem \ref{connectingsdc} 
 and  by \eqref{HcontainsgaK} since $\mathfrak{H}(\ga_1, x_1)=\mathfrak{H}(\ga, x_1)$.
 Vice versa if (a), (b) hold, let us define $\ga:=\eta\star \ga_1$; then (by definition of $E^K_{x_0,x_1}$)
 $\eta$ is a $SDC_K$; thus, $\ga$ is a $SDC_K$ too and $\ga_1$ is a $SDC_{co(K\cup\eta)}$ (see Remark \ref{sdcofprevious}).
\end{proof}

 \section{Self-distancing sets and steepest descent curves}\label{sigmaareingamma}.
 
A self-distancing 
 set $\si$ will be  called SDC-extendible  if there exists a steepest descent curve $\ga$ such that $\si \subset \ga$.
 
This section  is devoted to investigate the following question:\\
{\em Can a self-distancing set $\si$  be extended to a steepest descent curve $\ga$?} \\
Let us call $\Ga_\sigma$ the family of SDC  $\ga$ wich extends $\si$.
The following example shows that $\Ga_\sigma$ can be empty.
\begin{example}\label{si^4} Let us consider in a coordinate system xy the points:
$$x_1=(0,0), x_2=(0,2), x_3=(1,\sqrt{8}), x_4=(-1,\sqrt{8}).$$
The set $\tilde{\si}=\{x_i, i=1,\ldots, 4\}$ is a self-distancing  set not SDC-extendibile. 
\end{example}
\begin{proof}By contradiction let $\ga \in \Ga_{\tilde{\si}}$, then any point $x$ on the arc $\ga_{x_3,x_4}$ satisfies the inequalities
$$3=|x_3-x_1| \leq |x-x_1| \leq |x_4-x_1|=3, \quad  |x_3-x_2| \leq |x-x_2|.$$
That is $x\in \pa B(x_1,3)$ and $x\in \RR^2\setminus B(x_2, |x_3-x_2| ).$
Since $\pa B(x_1,3)\cap ( \RR^2\setminus B(x_2, |x_3-x_2| ))=\{x_3,x_4\}$, the arc $\ga_{x_3,x_4}$
 consists of two points only, that is impossible. 
\end{proof}

Next theorem gives a necessary condition \eqref{xincl(Jco(sigma))}
in order to  extend a finite self-distancing  set $\si$ to a 
SDC; this condition is  based on the bounding sets 
introduced in \S \ref{fencesintro}.

 Let us define   $\si_{x}$  as the subset of $\si$  consisting of the point $x$ and of 
the previous ones on $\si$ (consistently with  \eqref{defgax}).
\begin{theorem}\label{necesssiextendeinga}
 Let $\si$ be a self expanding SDC-extendible set,  then  for all $x_0\in \si$ such that $\si_{x_0}\neq \{x_0\}$, 
 the following 
 inclusion
 \begin{equation}\label{xincl(Jco(sigma))} 
  (\si\setminus \si_{x_0}) \subset  cl(\RR^2\setminus \mathfrak{J}(co(\si_{x_0}),x_0))
\end{equation}
holds.
\end{theorem}
\begin{proof}
 Let $\ga\in \Ga_\si$, then $\si \subset \ga$ and $\ga $  is a SDC. Then 
 $co(\ga_{x_0}) \supset co(\si_{x_0})$, $x_0\in \pa co(\ga_{x_0}) \cap \pa co(\si_{x_0})$ (see 
\cite[(i) of Lemma 4.6]{MLV});
 from Theorem \ref{incj} ,  
 $$\mathfrak{J}(co(\ga_{x_0}),x_0) \supset \mathfrak{J}(co(\si_{x_0}),x_0);$$
 moreover
  $\si\setminus \si_{x_0} \subset \ga \setminus \ga_{x_0}$ and from Corollary \ref{corjcogax_0}, 
 $$\ga\setminus \ga_{x_0} \subset cl(\RR^2\setminus \mathfrak{J}(co(\ga_{x_0}),x_0)).$$
 The previous inclusions prove \eqref{xincl(Jco(sigma))}. 
\end{proof}
\begin{remark} In the Example \ref{si^4} it has been proved, in a simple way,  that 
$\tilde{\si}$ is not SDC-extendible.
Another way to prove this fact is to check that  the  condition \eqref{xincl(Jco(sigma))} does not hold for the point
$x_4$; let us notice that $\pa \mathfrak{J}(co(\tilde{\si}_{x_3}),x_3)\cap \{x\leq 0, y\geq 0\} $ consists of a circular arc
centered at $x_1$ with radius $2+\sqrt{13-4\sqrt{8}}$; then 
it is easy to see that $x_4$ is in the interior of $\mathfrak{J}(co(\tilde{\si}_{x_3}),x_3)$ and \eqref{xincl(Jco(sigma))}
is not satisfied. 

Let us show in the following example 
that \eqref{xincl(Jco(sigma))} is not sufficient for a self-distancing set $\si$ to be SDC-extendible.

\begin{example}\label{si^4+} Let us consider in a coordinate system xy the points:
$$\xi_1=(0,0), \xi_2=(0,2), \xi_3=(2,0), \xi_4=(\rho,2 ).$$
For $\sqrt{8} <\rho< \pi$, 
the set $\si:=\{\xi_i, i=1,\ldots, 4\}$ is a self-distancing set satisfying the condition \eqref{xincl(Jco(sigma))}
not SDC-extendibile. 
\end{example}
\begin{proof} It easy to see that $\si$ is a self-distancing set. Moreover  the initial piece of the left involute
of $co(\{\xi_1,\xi_2,\xi_3\})$ starting at $\xi_3$ consists of a circular arc centered at $\xi_2$ of ray $\sqrt{8}$ and amplitude
$\frac34\pi$. Then $\xi_4 \not \in \mathfrak{J}(co(\{\xi_1,\xi_2,\xi_3\}),\xi_3)$ and 
\eqref{xincl(Jco(sigma))} is verified with $x_0=\xi_3$, $\si\setminus \si_{x_0}=\{\xi_4\}$. 
 Trivially \eqref{xincl(Jco(sigma))} is verified also at   $x_0=\xi_2 $.
Let us prove now that $\Ga_{\si}$ is empty.
By contradiction let $\ga \in \Ga_{\si}$.
Let us consider $\ga_{\xi_3}$. Since $\xi_2$, $\xi_3$ have the same distance from $\xi_1$, 
arguing as in Example \ref{si^4}, $\ga_{\xi_3}$ is a circular arc $\mathit{C}$
centered
at $\xi_1$ from $\xi_2$ to $\xi_3$. Since the arc $\ga_{\xi_3}$ has the distancing property from the segment $\xi_1\xi_2$,
it is necessarily the arc of amplitude $\pi/4$ and not the complementary arc  .
Let $\eta=\xi_1\xi_2\star C$.
Since  $\mathit{C} \subset \ga$ and $\xi_2, \xi_3 \in \ga$, then
$co(\eta) \subset \co(\ga_{\xi_3}).$ Thus by Corollary \ref{incj} 
$$\mathfrak{J}(co(\ga_{\xi_3}),\xi_3)\supset \mathfrak{J}(co(\eta),\xi_3) .$$
Since the segment $\xi_2\xi_4$ is tangent to $\eta $ at $\xi_2$ and it has length $\rho$, less than  $\pi$,
the length of the arc $\eta_{\xi_2,\xi_3}$, then
$$ \xi_4 \in Int(\mathfrak{J}(co(\ga_{\xi_3}),\xi_3)).$$
This is in contradiction with  Corollary \ref{corjcogax_0} at the point $\xi_3$.
 \end{proof}

Let us introduce  definitions and preliminary facts needed to obtain   necessary and sufficient conditions
for the extendibility of 
a self-distancing set $\si$ structured as follows.
\end{remark}

 \begin{definition}\label{deftildecup}
  Let us denote with $\tilde{\cup}_i\si_i$ 
  a self-distancing set with a finite (or countable) family of closed connected components 
 $\si_1, \si_2, \ldots, \si_n,\ldots $,  ordered as the points of $\si$, that 
 is if $i < j , x\in \si_i, y \in \si_j \Rightarrow x\preceq y  $. 
 Let $x_i^-$ be  the first point  and let $x_i^+$ be the  last point  of $ \si_i$;
 if they are distinct (that is $\si_i$  does not reduce to a point) 
 as noticed in the introduction (\cite[Theorem 3.3]{Daniilidis3} and \cite[Theorem 4.10]{MLV}), 
 $\si_i$ is a SDC and it will denoted by $\ga_i$.
 \end{definition}

 \begin{lemma}\label{Hinclusions} Let $\si=\tilde{\cup}_i\si_i$ be a self-distancing set.
 A necessary condition for $\Ga_\si \neq \emptyset$ is that  for all   components $ \si_i$,  which
 are curves $\ga_i$,
 the following fact 
 \begin{equation}\label{inclH}
  \quad \cup_{j=1}^{i} \si_j \subset \mathfrak{H}(\ga_{i+1}, x_{i+1}^-),
   \end{equation}
holds. 
 \end{lemma}
\begin{proof}
Let $\ga\in \Ga_\si$. Then  $\cup_{j=1}^{i} \si_j \subset \ga_{x_{i+1}^-}$.
Then \eqref{inclH} follows from  \eqref{HcontainsgaK}.
\end{proof}

\begin{definition}\label{defessential} Let $\si=\tilde{\cup}_i\si_i$ be a self-distancing set. 
A  subfamily  $E_\sigma\subset \Ga_\sigma $ is called {\em essential} for  $\Ga_\sigma $ if 
the following facts 
\begin{equation}\label{minimala}
 \mbox{a) \quad }\forall \ga \in \Ga_\sigma \; \exists\, \eta \in E_\sigma:  co(\eta) \subset co(\ga), 
\end{equation}
\begin{equation}\label{min|ga|} 
\mbox{b) \quad }\ga\in \Ga_\si \Rightarrow |\ga|\geq \min \{|\eta|, \eta \in E_{\si}  \}, 
\end{equation}
\begin{equation}\label{min|ga|=}
 \mbox{c) \quad }\ga\in \Ga_\si ,|\ga|= \min \{|\eta|, \eta \in E_{\si}  \}\Rightarrow \ga \in E_\si
\end{equation}
hold.

It can happen that $E_\si=\emptyset$. If $\Ga_\si =\emptyset$, let us define  $E_\sigma=\emptyset$ essential for $\Ga_\si$.
\end{definition}

Let us start to study a self-distancing set with two closed connected components.
\begin{lemma}\label{essential1-2}
 Let $\si^*=\tilde{\cup}_{i=1,2}\si_i$ be a self-distancing set and let $\rho$ be the segment joining $x_1^+,x_2^-$.
 There are five possibilities:
 \begin{enumerate}
  \item[p1)] Let $\si_1=\{x_1\}, \si_2=\{x_2\}$,
  then $E_{\si^*}=\{\rho\}\neq \emptyset$ is essential for $\Ga_\si$.
   \item[p2)]Let  $\si_1=\{x_1\}$, $\si_2=\ga_2$ ($\ga_2$ is a SDC); then a necessary 
   and sufficient condition for the extensibility of $ \si^*$ is 
  \begin{equation}\label{p3nc}
  \si_1 \subset  \mathfrak{H}(\ga_2, x_2^-);
  \end{equation}
   \end{enumerate}
moreover   $E_{\si^*}=\{\rho\star \ga_2\}$ is   essential    for $\Ga_{\si^*}$.
    \begin{enumerate}
  \item[p3)] Let $\si_1=\ga_1$ be a SDC, $\si_2=\{x_2\}$;  then a necessary 
   and sufficient condition for the extensibility of $ \si^*$ is 
  \begin{equation}\label{p2nc}
   \si_2 \subset cl(\RR^2\setminus (\mathfrak{J}(co(\ga_1),x_1^+)\cup co(\ga_1));
  \end{equation}
  moreover $E_{\si^*}=\{\ga_1\star \eta: \eta \in E_{x_1^+,x_2}^{co(\ga_1)}\}$ (see Definition \ref{minimalE}).
  \item[p4)] Let $\si_1=\ga_1 ,\si_2=\ga_2$; then a necessary 
   and sufficient condition for the extensibility of $ \si^*$ is that there exists a SDC $\eta$ such that:
    \begin{equation}\label{p4nc} 
\eta \in E^{co(\ga_1)}_{x_1^+, x_2^-} \mbox{\quad and \quad }  \si_1\cup \eta \subset  \mathfrak{H}(\ga_2, x_2^-);
  \end{equation}
   \end{enumerate}
 the related  essential family   is
$  E_{\si^*}=   \{\ga_1\star\eta \star\ga_2 : \eta \mbox{\, satisfies }   \eqref{p4nc}   \}$.
   \item[p5)] If in the cases p2),p3),p4)   the 
   corresponding necessary and sufficient conditions do not hold, then $$   E_{\si^*}= \Ga_{\si^*}=\emptyset .$$
  \end{lemma}
\begin{proof}
 case p1) is trivial; in the case p2) the inclusion \eqref{p3nc} 
 follows from Lemma \ref{Hinclusions} with $i=2$ in \eqref{inclH}. It is also trivial that it is  sufficient.
The case p3) follows from Theorem \ref{connectingsdc} with $K=co(\ga_1)$. The case p4) follows from Theorem 
\ref{connectingwithconditionH} with $\ga_2$ in place of $\ga_1$, $K=co(\ga_1)$, $x_1^+$ in place of $x_0$, 
$x_2^-$ in place of $x_1$.
 \end{proof}

An easy sufficient  condition to check if  $\si=\tilde{\cup}_i\si_i$ is extendible is the following
\begin{theorem}\label{suffcicentiforest} Let $\si=\tilde{\cup}_i\si_i$ be a self-distancing set. Let  $\si^{(i)}=\tilde{\cup}_{j=1}^{i} \si_j$.
If \eqref{inclH} and 

\begin{equation}\label{x^-inN}
  x^-_{i+1}\in N_{co(\si^{(i)})}(x^+_i), \quad \forall  i\geq 1 
\end{equation}
hold, then $\Ga_\si \neq \emptyset$ and $\overline{\ga}$, which  linearly and orderly connects  $\si_i$, $\si_{i+1}$ with 
the segments $x_i^+x^-_{i+1}$, is a SDC and it
has minimal length in $\Ga_\si$. 
\end{theorem}
\begin{proof}Let us argue by induction on the self-distancing set $\si^{(i)}$. 
The case $i=1$ is contained in Lemma \ref{essential1-2}, since the assumptions 
\eqref{inclH}  and \eqref{x^-inN} are enough to get the corresponding assumptions in the cases p1),p2),p3),p4).
Moreover in the case p4) the curve $\eta=x_1^+x_2^-$ is such that $\ga^{(2)}=\ga_1\star \eta\star\ga_2$ is the SDC  
of minimal length extending $\si^{(2)}$.

Let $\ga^{(i)}$ the curve 
of minimal length extending $\si^{(i)}$.
Since the normal sector to $co(\si^{(i)})$ at 
$x^+_i$ coincides with  the  sector $N$ in the proof of  Theorem \ref{connectingsdc}, with $x^+_i$ in place of $x_0$,
$x^-_{i+1}$ in place of $x_1$ then, the assumption \eqref{x^-inN} implies that the case 1. of the proof of Theorem
\ref{connectingsdc} occurs. It follows that
$$\ga^{(i+1)}=\ga^{(i)}\star x^+_ix^-_{i+1}\star\si_{i+1}$$
 is of minimal length in $\Ga_{\si^{(i+1)}}$. Then,
 $\overline{\ga}=\cup_i\ga^{(i+1)}$ is of minimal length in $\Ga_\si$.
\end{proof}
 \begin{rem} Since, as noticed in 
  \cite[II, Section 2]{Manselli-Pucci}, $\forall u,w \in \si_{x_i^+}$ the angle  $ux_i^+w$ has opening less than  $\pi/2$, thus
  the related  normal sector in  \eqref{x^-inN} has opening greater or equal than $\pi/2$; then to check that 
  $x^-_{i+1}$ satisfies \eqref{x^-inN},
 is   easier than to check that $x^-_{i+1}$ is outside of the $\mathfrak{F}$-fence as in \eqref{p2nc}.  
 \end{rem}

Lemma \ref{Hinclusions} and  Theorem \ref{suffcicentiforest} give  
only  necessary and  only sufficient conditions, respectively, for
the extensibility of self expanding sets. Let us give 
  definitions in order to get necessary and sufficient conditions.

\begin{definition}\label{defessenzialeiterativo} Let $\si=\tilde{\cup}_j\si_j$ be a self-distancing  set. 
Let $E_i, i=2, \ldots, n, ....$ be defined by induction as follows:
 
 $E_2$ is the essential family related to  $\tilde{\cup}_{j=1}^2\si_j$, as given by 
  Lemma \ref{essential1-2};
  if  $i\geq 2 $, the $E_{i+1}$  related to $\tilde{\cup}_{j=1}^{i+1}\si_j$  is defined as follows:
  \begin{enumerate}
   \item[i)] if $E_i=\emptyset$ then $E_{i+1}=\emptyset $;
   \item[ii)] if $E_i\neq \emptyset$, let us consider for all $\eta \in E_i$ the essential family $E(\eta)$
   (see Lemma \ref{essential1-2})
   related to
   $\eta \tilde{\cup} \si_{i+i}$ (see Definition \ref{deftildecup}). Let
   $E_{i+1}=\cup_{\eta \in E_i} E(\eta)$.
     \end{enumerate}
\end{definition}
Let us notice that $\{E_i\}$ is ordered by inclusion and $E_{i+1}$, if it is not empty, consists of $2^i$ curves at most.
\begin{theorem}\label{teoremafinalepersigmaexpanding} Let $\si=\tilde{\cup} \si_j$ be  a self  expanding set and  let  
 $E_2, E_3,\ldots ,E_i, \ldots $ be the sequence (finite or countable)  of the   essential families associated 
 to $\si$. 
 Then $\Ga_\si\neq \emptyset$ iff  $\forall i\geq 2$ the essential family $E_i$ is not empty.
 \end{theorem}
\begin{proof}If there exists $\ga\in \Ga_\si$  then, for all $i\geq 1$, 
$\ga_{x^+_{i+1}}\in \Ga_{\tilde{\cup}_{j=1}^{i+1}\si_j}$;
thus $E_{i+1}\neq \emptyset$ by Theorem \ref{connectingsdc}. Vice versa if at each step $i\geq 1$ the essential family 
$E_{i+1}\neq \emptyset$, then, by definition, there exists a sequence $\{\eta^{i+1}\}$ of SDC  such that 
$\eta^{i+1}\in E_{i+1}$ and such that $\eta^s \subset \eta^{s+1}, s\geq 1$ (that is, at each step, 
$\eta^{s+1}$ is an extension of a previous one $\eta^s$) and  $\eta^{s+1}\in E_{\tilde{\cup}_{j=1}^{i+1}\si_j}$, see 
Definition \ref{defessenzialeiterativo}. Then is well defined $\ga=\cup_{i=1}^\infty \eta^{i+1}$; obviously $\ga \in \Ga_\si$.
 \end{proof}


\vspace{1cm}

 {\bf Open problem}: In the present work only two dimensional problems are studied. In three (or more) dimensions 
 the construction of boundary regions to a SDC and to a $SDC_K$ is open. The boundary  regions should probably be constructed
by using  the space involutes of the geodesics curves on $\pa K$.
 

\end{document}